\documentclass[times,sort&compress,3p]{elsarticle}

\journal{Journal of Multivariate Analysis}

\usepackage{amsmath,amssymb,amsthm}
\usepackage[labelfont=bf]{caption} % DO NOT TOUCH (this is the standard for JMVA)
\usepackage{appendix} % To create appendix with proper labels.
\usepackage{enumitem} % This is needed to label and reference the assumptions correctly (the enumerate package cannot do the same).
\usepackage[colorlinks=true]{hyperref} % Do not put bookmarks = false, bookmarks are useful for navigating a paper.
\usepackage{centernot} % This is to get a larger slash than with the \not command. We use it in the equation H0 versus H1.
\usepackage{dsfont} % Best font for indicator functions.

\newtheorem{cor}{Corollary}
\newtheorem{definition}{Definition}
\newtheorem{lem}{Lemma}
\newtheorem{prp}{Proposition}
\newtheorem{rem}{Remark}
\newtheorem{thm}{Theorem}
\newcommand{\N}{\mathbb{N}}
\newcommand{\R}{\mathbb{R}}
\newcommand{\PP}{\mathsf{P}} % Russian style (do not change)
\newcommand{\EE}{\mathsf{E}} % Russian style (do not change)
\newcommand{\Var}{\mathsf{Var}} % Russian style (do not change)
\newcommand{\Cov}{\mathsf{Cov}} % Russian style (do not change)
\newcommand{\bb}[1]{\boldsymbol{#1}}
\newcommand{\ind}{\mathds{1}}
\newcommand{\nvert}[0]{\, \vert \, }
\newcommand{\rd}{{\rm d}}

\allowdisplaybreaks

\begin{document}

\begin{frontmatter}

\title{Asymptotics for non-degenerate multivariate \texorpdfstring{$U$}{U}-statistics with estimated nuisance parameters under the null and local alternative hypotheses}

\author[a1]{Alain Desgagn\'e\corref{mycorrespondingauthor}}
\author[a2]{Christian Genest}
\author[a2,a3]{Fr\'ed\'eric Ouimet}

\address[a1]{D\'epartement de math\'ematiques, Universit\'e du Qu\'ebec \`a Montr\'eal, Montr\'eal (Qu\'ebec) Canada H3C 3P8}
\address[a2]{Department of Mathematics and Statistics, McGill University, Montr\'eal (Qu\'ebec) Canada H3A 0B9}
\address[a3]{Centre de recherches math\'ematiques, Universit\'e de Montr\'eal, Montr\'eal (Qu\'ebec) Canada H3T 1J4\vspace{-7mm}}

\cortext[mycorrespondingauthor]{Corresponding author. Email address: desgagne.alain@uqam.ca}

\begin{abstract}
The large-sample behavior of non-degenerate multivariate $U$-statistics of arbitrary degree is investigated under the assumption that their kernel depends on parameters that can be estimated consistently. Mild regularity conditions are provided which guarantee that once properly normalized, such statistics are asymptotically multivariate Gaussian both under the null hypothesis and sequences of local alternatives. The work of Randles (1982, \emph{Ann. Statist.}) is extended in three ways: the data and the kernel values can be multivariate rather than univariate, the limiting behavior under local alternatives is studied for the first time, and the effect of knowing some of the nuisance parameters is quantified. These results can be applied to a broad range of goodness-of-fit testing contexts, as shown in two specific examples.
\end{abstract}


\end{frontmatter}

\section{Introduction}\label{sec:introduction}

Introduced by \citet{MR15746}, $U$-statistics are a key tool in statistics. They arise naturally in producing minimum-variance unbiased estimators for a given parameter functional in integral form. Many well-known statistics are either $U$-statistics or can be represented as linear combinations thereof. They find applications in many areas, including regression analysis and dependence modeling, and are closely related to the ubiquitous class of von Mises $V$-statistics. The reader is referred to the book by \citet{MR1075417} for an introduction to this topic.

\citet{MR26294} was the first to show that non-degenerate $U$-statistics are asymptotically Gaussian. Many authors subsequently obtained related limit theorems and approximations, such as Berry--Ess\'een theorems, Edgeworth expansions, large deviations or laws of the iterated logarithm, providing gradually relaxed moment conditions and improved error rates. Some of these issues have also been investigated in the degenerate case and under variants of the original definition, such as $U$-statistics that are trimmed, incomplete or have data dependencies (mixing conditions). Refer to \cite{MR1075417,MR1472486} for an account of classical results. For a more contemporary treatment of the theory, featuring examples from biomedical and psychosocial research and including discussion on missing data, refer to \cite{MR2368050}. In the dependent data setting, refer to \cite{MR1404665,MR1379055,MR1669980,MR2283250,MR2557623,MR2922461,MR2948902,MR2891437,MR3053547,MR3600463}.

However, relatively few papers considered the case of $U$-statistics whose kernel depends on nuisance parameters that must be estimated. This issue was first investigated for location parameters by \citet{MR95559}, who found necessary and sufficient conditions for the asymptotics to remain the same as in the known parameter case; he also extended his results to (non-degenerate) generalized $U$-statistics and functions of several generalized $U$-statistics. \citet{MR653521} later obtained normal asymptotics for the difference between non-degenerate $U$-statistics and their unknown means; his results required a certain smoothness and $L^1$ domination of the kernel function, as well as the condition that the nuisance parameter estimators satisfy a central limit theorem jointly with the $U$-statistic. \citet{MR653522} derived similar results in the related setting where the nuisance parameters in the asymptotic mean are also estimated.

The effect of estimating nuisance parameters on the large-sample behavior of degenerate $U$- and $V$-statistics of degree $2$ was first examined by \citet{MR936366}. This work was extended by \citet{MR885745} to encompass more general statistics that take the form of limiting chi-square (degenerate) $U$- or $V$-statistics of degree $1$ or $2$, whose limit is the same as a weighted sum of potentially infinitely many independent $\chi_1^2$ random variables. The latest advances are due to \citet{MR4411854}, who further broadened the scope of the results of \citet{MR885745} to include kernel functions of any degree which are not necessarily differentiable with respect to the nuisance parameters.

The purpose of this paper is to determine the asymptotic behavior, both under the null hypothesis and local alternatives, of a non-degenerate multivariate $U$-statistic with constant mean when nuisance parameters are estimated under a given null hypothesis. This extends the results of \citet{MR653521} in three directions: the output of the kernel and the observations are assumed to be multivariate instead of univariate, the asymptotics under local alternatives are investigated for the first time, and the effect of knowing a subset of the nuisance parameters is made explicit.

The necessary notations and assumptions are introduced in Section~\ref{sec:2} and Section~\ref{sec:3}, respectively. The two main results are then stated in Section~\ref{sec:4} and, to avoid disrupting the flow, their proofs are relegated to Section~\ref{sec:proofs}. These results are applied in Section~\ref{sec:5} to recover and extend a goodness-of-fit test developed by \citet{MR4547729} for the exponential power distribution (Example~1; Section~\ref{sec:example.1}) and to study the behavior of a statistic which is asymptotically equivalent to the Wilcoxon signed-rank test statistic, under a normality assumption where the location and variance parameters are unknown and must be estimated (Example~2; Section~\ref{sec:example.2}). A brief discussion ensues in Section~\ref{sec:6}. \ref{app:A} contains a technical lemma on a variant of the uniform law of large numbers (used in the proof of Theorem~\ref{thm:1}), as well as moment results for the asymmetric power distribution (needed for Example~1 in Section~\ref{sec:example.1}).

\section{Statement of the problem and notation}\label{sec:2}

Let $\bb{X}_1, \ldots, \bb{X}_n$ be a collection of mutually independent and identically distributed (iid) continuous random vectors taking values in $\R^m$ for some integer $m \in \N$. The problem of interest is that of testing
\[
\mathcal{H}_0 : \bb{X}_1, \ldots, \bb{X}_n \sim F(\cdot \nvert \bb{\theta}_0),
\]
where $\bb{\theta}_0$ is a fixed $p$-dimensional parameter vector, each of whose components can be either known or unknown. The cumulative distribution function $F(\cdot \nvert \bb{\theta}_0)$ is assumed to belong to a family of continuous distributions, denoted by $\{F(\cdot \nvert \bb{\theta}) : \bb{\theta} \in \Theta \}$, where $\Theta$ represents the parameter space. The latter is taken to be an open subset of $\R^p$, so that $\bb{\theta}_0$ is always an interior point of $\Theta$. For a general parameter $\bb{\theta} = (\theta_1, \ldots, \theta_p)^{\top}\!\!\in \Theta$, the density corresponding to $F(\cdot \nvert \bb{\theta})$ is denoted by $f(\cdot \nvert \bb{\theta})$.

The alternative hypothesis can be specified as $\mathcal{H}_1 : \bb{X}_1, \ldots, \bb{X}_n \centernot{\sim} F(\cdot \nvert \bb{\theta}_0)$ if an omnibus test is desired, or it can be refined to increase power based on prior knowledge. For instance, one could take $\mathcal{H}_1 : \bb{X}_1, \ldots, \bb{X}_n \sim F(\cdot \nvert \bb{\theta}), ~\bb{\theta} \neq \bb{\theta}_0$, if it is assumed that the true distribution belongs to the family $\{F(\cdot \nvert \bb{\theta}) : \bb{\theta} \in \Theta\}$. Another common example would be to specify in $\mathcal{H}_1$ that the true distribution is asymmetric (or symmetric). In any case, the results of this paper depend solely on the specification of $\mathcal{H}_0$ and the local alternatives, rather than on the specific form of $\mathcal{H}_1$.

Let $\nu\in \N$ be a positive integer which is smaller than $n$. Consider basing a test for $\mathcal{H}_0$ on a multivariate $U$-statistic of degree $\nu$, which is defined by
\[
\bb{U}_n(\bb{\theta}) = \binom{n}{\nu}^{-1} \sum_{(n, \nu)} \bb{h}(\bb{X}_{i_1}, \ldots \bb{X}_{i_{\nu}} \nvert \bb{\theta}),
\]
where the sum is taken over all subsets of indices $1 \leq i_1 < \cdots < i_{\nu} \leq n$, and the $d$-dimensional kernel
\[
\bb{h}(\bb{x}_1, \ldots , \bb{x}_{\nu} \nvert \bb{\theta})
=
\begin{bmatrix}
h_1(\bb{x}_1, \ldots, \bb{x}_{\nu} \nvert \bb{\theta}) \\
\vdots \\
h_d(\bb{x}_1, \ldots, \bb{x}_{\nu} \nvert \bb{\theta})
\end{bmatrix}
\]
is assumed to be symmetric, i.e., invariant under permutations of the arguments $\bb{x}_1, \ldots, \bb{x}_{\nu}$.

Throughout the paper, weak convergence and convergence in $\PP$-probability are denoted by $\stackrel{\PP}{\rightsquigarrow}$ and $\stackrel{\PP}{\rightarrow}$, respectively. When the underlying probability measure is clear from the context, it may be omitted. For a sequence $(\xi_n)_{n\in \N}$ of real-valued random variables and a sequence $(a_n)_{n\in \N}$ of real numbers, the notation $\xi_n = o_{\PP}(a_n)$ means that $\xi_n/ a_n \to 0$ in $\PP$-probability as $n\to \infty$, while $\xi_n = O_{\PP}(a_n)$ means that the sequence $(\xi_n/a_n)_{n\in \N}$ is tight, i.e., $\sup_{n\in \N} \PP(|\xi_n/a_n| \geq K)\to 0$ as $K\to \infty$. Also, for any integer $m\in \N$, let $\bb{0}_m = (0, \ldots, 0)^{\top}$, $\bb{1}_m = (1, \ldots, 1)^{\top}$, and $\mathrm{Id}_{m\times m} = \mathrm{diag}(1, \ldots, 1)$.

Let $\mathcal{K}\subseteq \{1, \ldots, p\}$ be the subset of indices for the $p_{\mathcal{K}}$ known components of $\bb{\theta}_0$, and let the complementary subset of indices $\mathcal{U} = \mathcal{K}^{\complement}$ correspond to the $p_{\mathcal{U}}$ unknown components. Define $\bb{\theta}_{0,\mathcal{K}} = (\theta_{0, j})_{j\in \mathcal{K}}$ and $\bb{\theta}_{0,\mathcal{U}} = (\theta_{0, j})_{j\in \mathcal{U}}$ as the subsets of $\bb{\theta}_0$ for known and unknown parameters, respectively. Throughout the paper, completely analogous notations are used for subsets of components of a general parameter $\bb{\theta}$ and any estimator of $\bb{\theta}_0$, i.e., $\bb{\theta}_{\mathcal{K}}$, $\bb{\theta}_{\mathcal{U}}$, $\hat{\bb{\theta}}_{n,\mathcal{K}}$, $\hat{\bb{\theta}}_{n,\mathcal{U}}$, etc.

The score functions with respect to $\bb{\theta}$, $\bb{\theta}_{\mathcal{K}}$, $\bb{\theta}_{\mathcal{U}}$ are respectively denoted by
\begin{equation}
\label{eq:score.function}
\bb{s}(\bb{x} \nvert \bb{\theta}_0) = \partial_{\bb{\theta}} \ln \{ f(\bb{x} \nvert\bb{\theta}_0)\}, \quad\bb{s}_{\mathcal{K}}(\bb{x} \nvert \bb{\theta}_0) = \partial_{\bb{\theta}_{\mathcal{K}}} \ln \{f(\bb{x} \nvert\bb{\theta}_0)\}, \quad \bb{s}_{\mathcal{U}}(\bb{x} \nvert \bb{\theta}_0) = \partial_{\bb{\theta}_{\mathcal{U}}} \ln \{ f(\bb{x} \nvert\bb{\theta}_0)\},
\end{equation}
and the corresponding expected information matrices, frequently called Fisher information matrices, are defined as
\[
I_{\bb{\theta}} = \EE\big\{\bb{s}(\bb{X} \nvert \bb{\theta}) \bb{s}(\bb{X} \nvert \bb{\theta})^{\top}\big\},\quad
I_{\bb{\theta},\mathcal{K}} = \EE\big\{\bb{s}_{\mathcal{K}}(\bb{X} \nvert \bb{\theta}) \bb{s}_{\mathcal{K}}(\bb{X} \nvert \bb{\theta})^{\top}\big\},\quad
I_{\bb{\theta},\mathcal{U}} = \EE\big\{\bb{s}_{\mathcal{U}}(\bb{X} \nvert \bb{\theta}) \bb{s}_{\mathcal{U}}(\bb{X} \nvert \bb{\theta})^{\top}\big\}
\]
and $ I_{\bb{\theta},\mathcal{K},\mathcal{U}} = \EE\big\{\bb{s}_{\mathcal{K}}(\bb{X} \nvert \bb{\theta}) \bb{s}_{\mathcal{U}}(\bb{X} \nvert \bb{\theta})^{\top}\big\}$, where $\bb{X} \sim F(\cdot \nvert \bb{\theta})$.

\section{Assumptions\label{sec:3}}

The assumptions required for the validity of the theorems stated in Section~\ref{sec:4} are as follows.

\begin{enumerate}[label=A\arabic*]\setlength\itemsep{0em}

\item
$ \EE\{\bb{h}(\bb{X}_1, \ldots, \bb{X}_{\nu} \nvert \bb{\theta})\} = \bb{0}_d$ for $\bb{X}_1, \ldots, \bb{X}_{\nu}\stackrel{\mathrm{iid}}{\sim} F(\cdot \nvert \bb{\theta})$ and all $\bb{\theta}\in \Theta$ such that $\bb{\theta}_{\mathcal{K}} = \bb{\theta}_{0,\mathcal{K}}$. \label{ass:1}

\item
The kernel $\bb{h}$ is non-degenerate, i.e., for $\bb{X}_1, \ldots, \bb{X}_n\stackrel{\mathrm{iid}}{\sim} F(\cdot \nvert \bb{\theta}_0)$, as $n \to \infty$,
\[
\sqrt{n} \, \bb{U}_n(\bb{\theta}_0) = \nu \frac{1}{\sqrt{n}} \sum_{i=1}^n \bb{h}^{(1)}(\bb{X}_i \nvert \bb{\theta}_0) + O_{\PP_{\!\mathcal{H}_0}}\left(\frac{1}{\sqrt{n}}\right) \bb{1}_d,
\]
where $\bb{h}^{(1)}(\bb{x} \nvert \bb{\theta}_0) = \EE\{\bb{h}(\bb{x}, \bb{X}_2, \ldots, \bb{X}_{\nu} \nvert \bb{\theta}_0)\}$ is a kernel of degree 1 and the matrix $ \EE\{\bb{h}^{(1)}(\bb{X} \nvert \bb{\theta}_0) \bb{h}^{(1)}(\bb{X} \nvert \bb{\theta}_0)^{\top}\}$ is positive definite. Note that when $\nu = 1$, the condition is automatically satisfied with $\bb{h}^{(1)}(\bb{x} \nvert \bb{\theta}_0) = \EE\{\bb{h}(\bb{x} \nvert \bb{\theta}_0)\} = \bb{h}(\bb{x} \nvert \bb{\theta}_0)$ and the $O_{\PP_{\!\mathcal{H}_0}}$ term above vanishes. \label{ass:2}

\item
For any $\delta\in (0, \infty)$, define the compact ball $B_{\delta}(\bb{\theta}_0) = \{\bb{\theta}\in \R^p : \|\bb{\theta} - \bb{\theta}_0\|_2 \leq \delta\}$, where $\|\cdot\|_2$ denotes the Euclidean norm. There exists $\delta\in (0, \infty)$ such that $B_\delta (\boldsymbol{\theta}_0)\subseteq \Theta$, and such that, for all $(i, j)\in \{1, \ldots, d\} \times \{1, \ldots, p\}$, one has
\begin{itemize}\setlength\itemsep{0em}
\item[(a)]
for all $\bb{x}_1, \ldots, \bb{x}_{\nu}\in \R^m$, the map $\bb{\theta}\mapsto \partial_{\theta_j} h_i(\bb{x}_1, \ldots, \bb{x}_{\nu} \nvert \bb{\theta})$ is continuous on $B_{\delta}(\bb{\theta}_0)$;

\item[(b)]
there exists a function $K_{ij} : \R^m \times \dots \times \R^m \to [0, \infty)$ such that $|\partial_{\theta_j} h_i(\bb{x}_1, \ldots, \bb{x}_{\nu} \nvert \bb{\theta})| \leq K_{ij}(\bb{x}_1, \ldots, \bb{x}_{\nu})$ for all $(\bb{x}_1, \ldots, \bb{x}_{\nu}, \bb{\theta})\in \R^m \times \dots \times \R^m \times B_{\delta}(\bb{\theta}_0)$, and
\[
\int_{\R^m} \cdots \int_{\R^m} K_{ij}(\bb{x}_1, \ldots, \bb{x}_{\nu}) \prod_{\ell=1}^{\nu} f(\bb{x}_{\ell} \nvert \bb{\theta}_0) \, \rd \bb{x}_1 \cdots \rd \bb{x}_{\nu} < \infty;
\]

\item[(c)]
there exists a function $L_j : \R^m \times \dots \times \R^m \to [0, \infty)$ such that $|\partial_{\theta_j} \textstyle\prod_{\ell=1}^{\nu} f(\bb{x}_{\ell} \nvert \bb{\theta})| \leq L_j(\bb{x}_1, \ldots, \bb{x}_{\nu})$ for all $(\bb{x}_1, \ldots, \bb{x}_{\nu}, \bb{\theta})\in \R^m \times \dots \times \R^m \times B_{\delta}(\bb{\theta}_0)$, and
\[
\int_{\R^m} \cdots \int_{\R^m} |h_i(\bb{x}_1, \ldots, \bb{x}_{\nu} \nvert \bb{\theta}_0)| L_j(\bb{x}_1, \ldots, \bb{x}_{\nu}) \, \rd \bb{x}_1 \cdots \rd \bb{x}_{\nu} < \infty.
\]
\end{itemize} \label{ass:3}

\item
There exists an $\R^p$-valued estimator $\hat{\bb{\theta}}_n = \hat{\bb{\theta}}_n(\bb{X}_1, \ldots, \bb{X}_n)$, with fixed known components $\hat{\bb{\theta}}_{n,\mathcal{K}} = \bb{\theta}_{0,\mathcal{K}}$, and an $\R^{p_{\mathcal{U}}}$-valued map $\bb{x}\mapsto \bb{r}_{\mathcal{U}}(\bb{x} \nvert \bb{\theta})$ such that $ \EE\{\bb{r}_{\mathcal{U}}(\bb{X} \nvert \bb{\theta}_0)\} = \bb{0}_{p_{\mathcal{U}}}$, the covariance matrix $R_{\bb{\theta}_0,\mathcal{U}} = \EE\big\{\bb{r}_{\mathcal{U}}(\bb{X} \nvert \bb{\theta}_0)$ $\bb{r}_{\mathcal{U}}(\bb{X} \nvert \bb{\theta}_0)^{\top}\big\}$ is positive definite for $\bb{X}\sim F(\cdot \nvert \bb{\theta}_0)$, and for $\smash{\bb{X}_1, \ldots, \bb{X}_n\stackrel{\mathrm{iid}}{\sim} F(\cdot \nvert \bb{\theta}_0)}$, the following expansion holds as $n\to \infty$:
\[
\sqrt{n} \, (\hat{\bb{\theta}}_{n,\mathcal{U}} - \bb{\theta}_{0,\mathcal{U}}) = R_{\bb{\theta}_0,\mathcal{U}}^{-1} \frac{1}{\sqrt{n}} \sum_{i=1}^n \bb{r}_{\mathcal{U}}(\bb{X}_i \nvert \bb{\theta}_0) + o_{\hspace{0.3mm}\PP_{\!\mathcal{H}_0}}(1) \bb{1}_{p_{\mathcal{U}}}.
\]
\label{ass:4}
\end{enumerate}

A few comments about these assumptions are in order. First, Assumption~\ref{ass:1} is not restrictive as the kernel function $\bb{h}$ in $\bb{U}_n(\bb{\theta})$ can always be replaced by the new kernel function $\tilde{\bb{h}}(\bb{x}_1, \ldots, \bb{x}_{\nu} \nvert \bb{\theta}) = \bb{h}(\bb{x}_1, \ldots, \bb{x}_{\nu} \nvert \bb{\theta}) - \EE\{\bb{h}(\bb{X}_1, \ldots, \bb{X}_{\nu} \nvert \bb{\theta})\}$.

Assumption~\ref{ass:2} implies that $\!\sqrt{n} \, \bb{U}_n(\bb{\theta}_0)$ is asymptotically normal. The non-degeneracy of the kernel is a standard assumption, which was also used by \citet{MR653521}. The degenerate case  is more difficult and left for future research.

Assumption~\ref{ass:3} is there for multiple reasons. Conditions (b) and (c) ensure that dominated convergence arguments hold; they are used to interchange partial derivatives and integral in the proof of Theorem~\ref{thm:1} to find an asymptotic expression for the first order derivative term in the stochastic Taylor expansion of $\bb{U}_n(\hat{\bb{\theta}}_n)$ around $\bb{U}_n(\bb{\theta}_0)$. Condition (b) is also used, together with (a), to control the error term in the same expansion by an adaptation of a well-known uniform law of large numbers due to Le Cam \citep{LeCam1952phd}, whose proof relies on dominated convergence and the continuity of $\bb{\theta}\mapsto \EE\{\partial_{\theta_j} h_i(\bb{X}_1, \ldots, \bb{X}_{\nu} \nvert \bb{\theta})\}$.

As pointed just above Lemma~7.6 of \citet{MR1652247}, condition (c) also implies that the density $ f( \cdot \nvert \bb{\theta})$ is differentiable in quadratic mean at the interior point $\bb{\theta}_0\in \Theta\subseteq \R^p$. This fact is also used in the proof of Theorem~\ref{thm:2} to derive a second order expansion of the log-ratio between the densities of the observations under $\mathcal{H}_0$ and local alternatives. This key argument, together with the asymptotics of Theorem~\ref{thm:1} under $\mathcal{H}_0$, ultimately leads to the derivation of similar asymptotics under local alternatives in Theorem~\ref{thm:2}.

As mentioned in Remark~1 of \citet{MR1146353}, Assumption~\ref{ass:4}, which is due to \citet{MR451521}, is satisfied by many types of estimators (provided that certain regularity conditions are met): method of moments, maximum likelihood and minimum chi-squared estimators, as well as sample quantiles, among others. It automatically yields $\hat{\bb{\theta}}_{n,\mathcal{U}} \to \bb{\theta}_{0,\mathcal{U}}$ in $\PP_{\!\mathcal{H}_0}$-probability as $n\to \infty$. Further, as shown, e.g., in Chapter~5 of \citet{MR1652247}, this type of expansion holds for the class of $M$-estimators under mild regularity conditions. In particular, for maximum likelihood estimators, Assumption~\ref{ass:4} is usually satisfied by the score function by setting $\bb{r}_{\mathcal{U}}(\bb{x} \nvert \bb{\theta}) = \bb{s}_{\mathcal{U}}(\bb{x} \nvert \bb{\theta})$ and $R_{\bb{\theta},\mathcal{U}} = I_{\bb{\theta},\mathcal{U}}$. Indeed, if the density $f$ is smooth enough for the Taylor expansion below to hold (refer to Section~5.3 of \citet{MR1652247} for specific conditions), one has
\begin{align*}
\frac{1}{\sqrt{n}} \sum_{i=1}^n \partial_{\bb{\theta}_{\mathcal{U}}} \ln \{ f(\bb{X}_i \nvert\hat{\bb{\theta}}_n)\}
&= \frac{1}{\sqrt{n}} \sum_{i=1}^n \partial_{\bb{\theta}_{\mathcal{U}}} \ln \{ f(\bb{X}_i \nvert\bb{\theta}_0) \} + \left[ \frac{1}{n} \sum_{i=1}^n \partial_{\bb{\theta}_{\mathcal{U}}^{\top}}\partial_{\bb{\theta}_{\mathcal{U}}} \ln \{ f(\bb{X}_i \nvert \bb{\theta}_0) \} \right] \sqrt{n} \, (\hat{\bb{\theta}}_{n,\mathcal{U}} - \bb{\theta}_{0,\mathcal{U}}) + o_{\hspace{0.3mm}\PP_{\!\mathcal{H}_0}}(1) \bb{1}_{p_{\mathcal{U}}} \\
\Leftrightarrow~~
\bb{0}_{p_{\mathcal{U}}}
&= \frac{1}{\sqrt{n}} \sum_{i=1}^n \bb{s}_{\mathcal{U}}(\bb{X}_i \nvert \bb{\theta}_0) + \big\{-I_{\bb{\theta}_0,\mathcal{U}} + o_{\hspace{0.3mm}\PP_{\!\mathcal{H}_0}}(1) \bb{1}_{p_{\mathcal{U}}} \bb{1}_{p_{\mathcal{U}}}^{\top}\big\} \sqrt{n} \, (\hat{\bb{\theta}}_{n,\mathcal{U}} - \bb{\theta}_{0,\mathcal{U}}) + o_{\hspace{0.3mm}\PP_{\!\mathcal{H}_0}}(1) \bb{1}_{p_{\mathcal{U}}} \\
&\Leftrightarrow~
\sqrt{n} \, (\hat{\bb{\theta}}_{n,\mathcal{U}}-\bb{\theta}_{0,\mathcal{U}}) = I_{\bb{\theta}_0,\mathcal{U}}^{-1} \frac{1}{\sqrt{n}} \sum_{i=1}^n \bb{s}_{\mathcal{U}}(\bb{X}_i \nvert \bb{\theta}_0) + o_{\hspace{0.3mm}\PP_{\!\mathcal{H}_0}}(1) \bb{1}_{p_{\mathcal{U}}}.
\end{align*}

In some practical contexts, it may be useful to state Assumption~\ref{ass:4} in an alternative way, which appears for example as Condition~2 of \citet{MR1146353}, i.e., there exists an $\R^{p_{\mathcal{U}}}$-valued map $\bb{x}\mapsto \bb{g}_{\mathcal{U}}(\bb{x} \nvert \bb{\theta})$ such that $ \EE\big\{\bb{g}_{\mathcal{U}}(\bb{X} \nvert \bb{\theta}_0)\big\} = \bb{0}_{p_{\mathcal{U}}}$, the covariance matrix $ \EE\big\{\bb{g}_{\mathcal{U}}(\bb{X} \nvert \bb{\theta}_0) \bb{g}_{\mathcal{U}}(\bb{X} \nvert \bb{\theta}_0)^{\top}\big\} = R_{\bb{\theta}_0,\mathcal{U}}^{-1}$ is positive definite for $\bb{X}\sim F(\cdot \nvert \bb{\theta}_0)$, and for $\smash{\bb{X}_1, \ldots, \bb{X}_n\stackrel{\mathrm{iid}}{\sim} F(\cdot \nvert \bb{\theta}_0)}$, the following expansion holds as $n\to \infty$:
\[
\sqrt{n} \, (\hat{\bb{\theta}}_{n,\mathcal{U}} - \bb{\theta}_{0,\mathcal{U}}) = \frac{1}{\sqrt{n}} \sum_{i=1}^n \bb{g}_{\mathcal{U}}(\bb{X}_i \nvert \bb{\theta}_0) + o_{\hspace{0.3mm}\PP_{\!\mathcal{H}_0}}(1) \bb{1}_{p_{\mathcal{U}}}.
\]
The correspondence is seen by setting $\bb{r}_{\mathcal{U}}(\bb{x} \nvert \bb{\theta}_0)=R_{\bb{\theta}_0,\mathcal{U}}\bb{g}_{\mathcal{U}}(\bb{x} \nvert \bb{\theta}_0)$.

\section{Asymptotic distribution of the multivariate \texorpdfstring{$U$}{U}-statistic\label{sec:4}}

This section contains the statements of the two main results. The first one (Theorem~\ref{thm:1}) gives the asymptotic distribution of the multivariate $U$-statistic $\bb{U}_n(\hat{\bb{\theta}}_n)$ under the null hypothesis, for any combination of components of $\bb{\theta}_0$ that are either known or unknown. The second (Theorem~\ref{thm:2}) provides its asymptotic distribution under the sequence $\mathcal{H}_{1,n}(\bb{\delta}_{\mathcal{K}})$ of local alternatives in the same framework.

In the following, for all $\bb{\theta}\in \Theta$ and $\bb{X}\sim F(\cdot \nvert \bb{\theta})$, let
\begin{equation}\label{eq:matrix.defs}
\begin{aligned}
&\Sigma_{\bb{\theta},\mathcal{U}} = \nu^2 \big\{H_{\bb{\theta}} - G_{\bb{\theta},\mathcal{U}} R_{\bb{\theta},\mathcal{U}}^{-1} J_{\bb{\theta},\mathcal{U}}^{\top} - J_{\bb{\theta},\mathcal{U}} R_{\bb{\theta},\mathcal{U}}^{-1} G_{\bb{\theta},\mathcal{U}}^{\top} + G_{\bb{\theta},\mathcal{U}} R_{\bb{\theta},\mathcal{U}}^{-1} G_{\bb{\theta},\mathcal{U}}^{\top}\big\}, \\[0.5mm]
&H_{\bb{\theta}} = \EE\big\{\bb{h}^{(1)}(\bb{X} \nvert \bb{\theta}) \bb{h}^{(1)}(\bb{X} \nvert \bb{\theta})^{\top}\big\}, \quad
G_{\bb{\theta},\mathcal{U}} = \EE\big\{\bb{h}^{(1)}(\bb{X} \nvert \bb{\theta}) \bb{s}_{\mathcal{U}}(\bb{X} \nvert \bb{\theta})^{\top}\big\}, \\
&J_{\bb{\theta},\mathcal{U}} = \EE\big\{\bb{h}^{(1)}(\bb{X} \nvert \bb{\theta}) \bb{r}_{\mathcal{U}}(\bb{X} \nvert \bb{\theta})^{\top}\big\}, \quad
R_{\bb{\theta},\mathcal{U}} = \EE\big\{\bb{r}_{\mathcal{U}}(\bb{X} \nvert \bb{\theta}) \bb{r}_{\mathcal{U}}(\bb{X} \nvert \bb{\theta})^{\top}\big\}.
\end{aligned}
\end{equation}

\begin{thm}
\label{thm:1}
Suppose that Assumptions~\ref{ass:1}--\ref{ass:4} hold. Then, under $\mathcal{H}_0$ and as $n \to \infty$,
\[
\sqrt{n} \, \Sigma_{\bb{\theta}_0,\mathcal{U}}^{-1/2} \bb{U}_n(\hat{\bb{\theta}}_n)\rightsquigarrow \mathcal{N}_d(\bb{0}_d, \mathrm{Id}_{d\times d}),
\quad
n \, \bb{U}_n(\hat{\bb{\theta}}_n)^{\top} \Sigma_{\bb{\theta}_0,\mathcal{U}}^{-1} \bb{U}_n(\hat{\bb{\theta}}_n) \rightsquigarrow \chi_d^2.
\]
If in addition the matrix function $\Sigma_{\bb{\theta},\mathcal{U}}$ is almost-everywhere continuous in $\bb{\theta}$, then one has, as $n \to \infty$,
\[
\sqrt{n} \, \Sigma_{\hat{\bb{\theta}}_n,\mathcal{U}}^{-1/2} \bb{U}_n(\hat{\bb{\theta}}_n)\rightsquigarrow \mathcal{N}_d(\bb{0}_d, \mathrm{Id}_{d\times d}),
\quad
n \, \bb{U}_n(\hat{\bb{\theta}}_n)^{\top} \Sigma_{\hat{\bb{\theta}}_n,\mathcal{U}}^{-1} \bb{U}_n(\hat{\bb{\theta}}_n) \rightsquigarrow \chi_d^2.
\]
\end{thm}

\begin{rem}\label{rem:1}\upshape
The fact that the parameters $\bb{\theta}_{0,\mathcal{K}}$ are known has no impact on the asymptotics in Theorem~\ref{thm:1}. Hence these parameters could be considered implicit in the model. Nevertheless, they appear explicitly above because they have an impact in Theorem~\ref{thm:2}.
\end{rem}

\begin{rem}\label{rem:2}\upshape
When all the components of $\bb{\theta}_0$ are known, i.e., $\mathcal{U} = \emptyset$, one has $\Sigma_{\bb{\theta},\mathcal{U}} = \nu^2 H_{\bb{\theta}}$.
\end{rem}

\begin{rem}\label{rem:3}\upshape
If the maximum likelihood estimator satisfies Assumption~\ref{ass:4} with the score function for all unknown parameters, then one has $\bb{r}_{\mathcal{U}}(\bb{x} \nvert \bb{\theta}) = \bb{s}_{\mathcal{U}}(\bb{x} \nvert \bb{\theta})$, $R_{\bb{\theta},\mathcal{U}} = I_{\bb{\theta},\mathcal{U}}$ and $J_{\bb{\theta},\mathcal{U}} = G_{\bb{\theta},\mathcal{U}}$. Consequently, the asymptotic covariance matrix of $\sqrt{n} \, \bb{U}_n(\hat{\bb{\theta}}_n)$ reduces to $\Sigma_{\bb{\theta},\mathcal{U}} = \nu^2 (H_{\bb{\theta}} - G_{\bb{\theta},\mathcal{U}} I_{\smash{\bb{\theta},\mathcal{U}}}^{-1} G_{\smash{\bb{\theta},\mathcal{U}}}^{\top})$.
\end{rem}

\begin{rem}\label{rem:4}\upshape
The dominance conditions stated in items (b) and (c) of Assumptions~\ref{ass:3} are not always necessary to achieve asymptotic normality in Theorem~\ref{thm:1}. For instance, \citet{MR4547729} defined $\bb{U}_n$ to be the score statistic when testing the Laplace distribution with unknown location and scale parameters against the larger asymmetric power distribution (APD) family introduced by \citet{MR2395888}. The location and scale parameters were estimated using the maximum likelihood estimators, so an expansion like the one in Assumption~\ref{ass:4} did hold.

A result in the same vein as Theorem~\ref{thm:1}, but much less general, was given in Theorem~3.3 of \citet{MR4547729}. However, the proof required a few tweaks. When calculating the matrix $\bb{U}_n'$ in their proof to apply a variant of the uniform law of large numbers due to Le\,Cam, cf.\  \eqref{eq:expansion}, one of the components of $\bb{U}_n'$ had logarithmic summands of the form $x\mapsto \ln|x - \mu|$ that made the dominance conditions in Assumption~\ref{ass:3} impossible. This is because the envelope of the class of functions $\{x\mapsto \ln|x - \mu|\}_{|\mu| < \delta}$ is infinite in any small neighborhood of $x = \mu$. The saving grace is that while $\ln|\cdot|$ blows up at $0$, it is still integrable locally at $0$, and the tail of $\ln|\cdot|$ grows slowly enough at infinity compared to the exponential tail of the Laplace distribution.

It was essentially shown earlier by \citet{MR3842623} that the conditions of uniform laws of large numbers similar to the one stated in Lemma~\ref{lem:1} of \ref{app:A} could be tweaked to include this kind of misbehavior. Given that Assumptions~\ref{ass:1}--\ref{ass:4} already cover a very large range of practical cases, and considering that the adjustments that would be required to include this kind of misbehavior seems too specialized, this research direction is not pursued further here.
\end{rem}

By leveraging the asymptotics outlined in Theorem~\ref{thm:1}, it becomes straightforward to design a goodness-of-fit test for $\mathcal{H}_0$. The choice of a meaningful kernel for the $U$-statistic would be guided in general by the characteristics of the alternative hypothesis $\mathcal{H}_1$. More precisely, the null hypothesis $\mathcal{H}_0$ should be rejected if the observed value of the statistic $n \, \bb{U}_n(\hat{\bb{\theta}}_n)^{\top} \Sigma_{\smash{\hat{\bb{\theta}}_n, \mathcal{U}}}^{-1} \bb{U}_n(\hat{\bb{\theta}}_n)$ falls outside the interval delimited by the designated quantiles of the \(\chi_d^2\) distribution for a specified confidence level.

One could also be interested in calculating the power of that test, namely the probability of correctly rejecting $\mathcal{H}_0$ when a specific alternative hypothesis $\mathcal{H}_1$ represents the true distribution. As the asymptotic power of a well-designed test tends to $1$ as $n \to \infty$ for any fixed alternative different from the null hypothesis, the following set of local alternative hypotheses is considered instead, viz.
\[
\mathcal{H}_{1,n}(\bb{\delta}_{\mathcal{K}}) : \bb{X}_1, \ldots, \bb{X}_n \sim F(\cdot \nvert \bb{\theta}_n),
\]
where, for arbitrary integer $n \in \N$,
\[
\bb{\theta}_{n,\mathcal{K}} = \bb{\theta}_{0,\mathcal{K}} + \frac{\bb{\delta}_{\mathcal{K}}}{\sqrt{n}} \, \{1 + o(1)\}, \quad \bb{\theta}_{n,\mathcal{U}} = \bb{\theta}_{0,\mathcal{U}},
\]
and the real vector $\bb{\delta}_{\mathcal{K}}\in \R^{p_{\mathcal{K}}} \backslash \{\bb{0}_{p_{\mathcal{K}}}\}$ is known and fixed.

\begin{thm}\label{thm:2}
Suppose that Assumptions~\ref{ass:1}--\ref{ass:4} hold, and recall the notations in \eqref{eq:matrix.defs}. Then, under $\mathcal{H}_{1,n}(\bb{\delta}_{\mathcal{K}})$ and as $n \to \infty$,
\[
\sqrt{n} \, \Sigma_{\bb{\theta}_0,\mathcal{U}}^{-1/2} \bb{U}_n (\hat{\bb{\theta}}_n)
\rightsquigarrow \mathcal{N}_d\big(\Sigma_{\bb{\theta}_0,\mathcal{U}}^{-1/2} \, M_{\bb{\theta}_0}\bb{\delta}_{\mathcal{K}}, \mathrm{Id}_{d\times d}\big),
\quad
n \, \bb{U}_n(\hat{\bb{\theta}}_n)^{\top}\Sigma_{\bb{\theta}_0,\mathcal{U}}^{-1} \bb{U}_n(\hat{\bb{\theta}}_n) \rightsquigarrow \chi_d^2\big(\bb{\delta}_{\mathcal{K}}^{\top} M_{\bb{\theta}_0}^{\top} \Sigma_{\bb{\theta}_0,\mathcal{U}}^{-1} \, M_{\bb{\theta}_0}\bb{\delta}_{\mathcal{K}}\big),
\]
where $\chi^2_d (\zeta)$ refers to the chi-square distribution with $d$ degrees of freedom and noncentrality parameter $\zeta$ while
\begin{gather*}
M_{\bb{\theta}_0} = \nu \, \big(G_{\bb{\theta}_0,\mathcal{K}} - G_{\bb{\theta}_0,\mathcal{U}} R_{\bb{\theta}_0,\mathcal{U}}^{-1} S_{\!\bb{\theta}_0,\mathcal{K}, \mathcal{U}}^{\top}\big), \\
G_{\bb{\theta}_0,\mathcal{K}} = \EE\big\{\bb{h}^{(1)}(\bb{X} \nvert \bb{\theta}_0) \bb{s}_{\mathcal{K}}(\bb{X} \nvert \bb{\theta}_0)^{\top}\big\}, \quad
S_{\!\bb{\theta}_0,\mathcal{K}, \mathcal{U}} = \EE\big\{\bb{s}_{\mathcal{K}}(\bb{X} \nvert \bb{\theta}_0) \bb{r}_{\mathcal{U}}(\bb{X} \nvert \bb{\theta}_0)^{\top}\big\},
\end{gather*}
with the expectation taken with respect to $\bb{X}\sim F(\cdot \nvert \bb{\theta}_0)$. If in addition the matrix function $\Sigma_{\bb{\theta},\mathcal{U}}$ is almost-everywhere continuous in $\bb{\theta}$, then one has, as $n \to \infty$,
\[
\sqrt{n} \, \Sigma_{\hat{\bb{\theta}}_n, \mathcal{U}}^{-1/2} \bb{U}_n(\hat{\bb{\theta}}_n)
\rightsquigarrow \mathcal{N}_d\big(\Sigma_{\bb{\theta}_0,\mathcal{U}}^{-1/2} \, M_{\bb{\theta}_0}\bb{\delta}_{\mathcal{K}}, \mathrm{Id}_{d\times d}\big),
\quad
n \, \bb{U}_n(\hat{\bb{\theta}}_n)^{\top}\Sigma_{\hat{\bb{\theta}}_n, \mathcal{U}}^{-1} \bb{U}_n(\hat{\bb{\theta}}_n) \rightsquigarrow \chi_d^2\big(\bb{\delta}_{\mathcal{K}}^{\top} M_{\bb{\theta}_0}^{\top} \Sigma_{\bb{\theta}_0,\mathcal{U}}^{-1} \, M_{\bb{\theta}_0}\bb{\delta}_{\mathcal{K}}\big).
\]
\end{thm}

\begin{rem}\label{rem:5}\upshape
Remarks~\ref{rem:2}--\ref{rem:4} apply to Theorem~\ref{thm:2} as well as to Theorem~\ref{thm:1}. In particular, the local alternative version of the result mentioned in Remark~\ref{rem:4} for testing the Laplace distribution is given in Theorem~3.10 of \citet{MR4547729}.
\end{rem}

\begin{rem}\label{rem:6}\upshape
If the maximum likelihood estimator satisfies Assumption~\ref{ass:4} with the score function for all unknown parameters, then one has $\bb{r}_{\mathcal{U}}(\bb{x} \nvert \bb{\theta}) = \bb{s}_{\mathcal{U}}(\bb{x} \nvert \bb{\theta})$ and
\[
M_{\bb{\theta}_0}
= \nu \, \big(G_{\bb{\theta}_0,\mathcal{K}} - G_{\bb{\theta}_0,\mathcal{U}} I_{\bb{\theta}_0,\mathcal{U}}^{-1} I_{\!\bb{\theta}_0,\mathcal{K}, \mathcal{U}}^{\top}\big).
\]
\end{rem}

The shrewd reader might wonder what happens more generally in Theorem~\ref{thm:2} if the parameters $\bb{\theta}_{n,\mathcal{U}}$ are allowed to wiggle away from the null hypothesis just as the components of $\bb{\theta}_{n,\mathcal{K}}$ are under $\mathcal{H}_{1,n}(\bb{\delta}_{\mathcal{K}})$. In that case, the local alternative hypotheses would be formulated, for each integer $n \in \mathbb{N}$, as
\[
\mathcal{H}_{1,n}(\bb{\delta}) : \bb{X}_1, \ldots, \bb{X}_n \sim F(\cdot \nvert \bb{\theta}_n), \quad \bb{\theta}_n = \bb{\theta}_0 + \frac{\bb{\delta}}{\sqrt{n}} \, \{1 + o(1)\},
\]
where the real vector $\bb{\delta}\in \R^p \backslash \{\bb{0}_p\}$ is known and fixed. It turns out that under mild regularity conditions on the maximum likelihood estimator for the unknown components of $\bb{\theta}_0$, which are almost always satisfied for common distributions, it is in fact redundant to consider this broader case because the asymptotic distribution of $\sqrt{n} \, \bb{U}_n (\hat{\bb{\theta}}_n)$ under $\mathcal{H}_{1,n}(\bb{\delta})$ does not depend on the components $\bb{\delta}_{\mathcal{U}}$ of $\bb{\delta}$. This claim is formalized in Corollary~\ref{cor:1} below, and the proof is deferred to Section~\ref{sec:proofs}.

\begin{cor}\label{cor:1}
Suppose that Assumptions~\ref{ass:1}--\ref{ass:4} hold, and recall the notations in \eqref{eq:matrix.defs}. Also, assume that the maximum likelihood estimator $\bb{\theta}_n^{\star} = \bb{\theta}_n^{\star}(\bb{X}_1,\ldots,\bb{X}_n)$ satisfies the conditions of Assumptions~\ref{ass:4}, meaning that for the score function $\bb{s}_{\mathcal{U}}(\bb{x} \nvert \bb{\theta}_0) = \partial_{\bb{\theta}_{\mathcal{U}}} \ln \{ f(\bb{x} \nvert\bb{\theta}_0) \}$ defined in \eqref{eq:score.function}, one has $ \EE\{\bb{s}_{\mathcal{U}}(\bb{X} \nvert \bb{\theta}_0)\} = \bb{0}_{p_{\mathcal{U}}}$, the expected information matrix $I_{\bb{\theta}_0,\mathcal{U}} = \EE\big\{\bb{s}_{\mathcal{U}}(\bb{X} \nvert \bb{\theta}_0) \bb{s}_{\mathcal{U}}(\bb{X} \nvert \bb{\theta}_0)^{\top}\big\}$ is positive definite for $\bb{X}\sim F(\cdot \nvert \bb{\theta}_0)$, and for $\smash{\bb{X}_1, \ldots, \bb{X}_n\stackrel{\mathrm{iid}}{\sim} F(\cdot \nvert \bb{\theta}_0)}$, the following expansion holds as $n \to \infty$:
\[
\sqrt{n} \, (\bb{\theta}_{\smash{n,\mathcal{U}}}^{\star} - \bb{\theta}_{0,\mathcal{U}}) = I_{\bb{\theta}_0,\mathcal{U}}^{-1} \frac{1}{\sqrt{n}} \sum_{i=1}^n \bb{s}_{\mathcal{U}}(\bb{X}_i \nvert \bb{\theta}_0) + o_{\hspace{0.3mm}\PP_{\!\mathcal{H}_0}}(1) \bb{1}_{p_{\mathcal{U}}}.
\]
Then, under $\mathcal{H}_{1,n}(\bb{\delta})$, the conclusions are identical to those of Theorem~\ref{thm:2}.
\end{cor}

\section{Examples}\label{sec:5}

To illustrate the results derived in this article, two examples are provided; see Section~\ref{sec:example.1} and Section~\ref{sec:example.2}, respectively. The proofs of the propositions stated in these examples are deferred to Section~\ref{sec:proofs}.

\subsection{Example~1: Goodness-of-fit test for the exponential power distribution}
\label{sec:example.1}

In this section, the modified score goodness-of-fit test for the exponential power distribution (EPD), introduced by \citet{MR4547729,MR4512291}, is revisited. Drawing on the findings presented in Section~\ref{sec:4}, the asymptotic behavior of the test under both the null hypothesis and local alternatives is recovered in Proposition~\ref{prp:1} when the nuisance parameters are estimated via maximum likelihood; recall Remarks~\ref{rem:4}~and~\ref{rem:5}. These results are then extended in Proposition~\ref{prp:2} to the case where the nuisance parameters are estimated via the method of moments.

\begin{definition}\upshape
For an asymmetry parameter $\theta_1\in (0,1)$, a tail decay parameter $\theta_2\in (0,\infty)$, a location parameter $\mu\in \R$, a scale parameter $\sigma\in (0,\infty)$, and a given positive scalar $\lambda\in (0, \infty)$, define the density of the asymmetric power distribution with $4$-dimensional parameter vector $\bb{\theta} = (\theta_1, \theta_2, \mu, \sigma)^{\top}$ at any $x \in \R$ by
\begin{equation}
\label{eq:APD.density}
f_{\lambda}(x \nvert \theta_1,\theta_2,\mu,\sigma)
= \frac{(\delta_{\theta_1,\theta_2}/\lambda)^{1/\theta_2}}{\sigma\Gamma(1 + 1/\theta_2)}
\exp\left\{-\frac{1}{\lambda}\frac{\delta_{\theta_1,\theta_2}} {A_{\theta_1,\theta_2}(y)} \, |y|^{\theta_2}\right\},
\end{equation}
where $y = (x - \mu) /\sigma \in \R$,
\[
\delta_{\theta_1,\theta_2} = \frac{2 \theta_1^{\theta_2}
(1 - \theta_1)^{\theta_2}}{\theta_1^{\theta_2} + (1 - \theta_1)^{\theta_2}}\in (0,1),
\quad
A_{\theta_1,\theta_2}(y)
= \left\{1/2 + \mathrm{sign}(y) (1/2 - \theta_1)\right\}^{\theta_2}
=
\begin{cases}
\theta_1^{\theta_2} &\mbox{if } y\in (-\infty,0), \\
(1/2)^{\theta_2} &\mbox{if } y = 0, \\
(1-\theta_1)^{\theta_2} &\mbox{if } y\in (0,\infty).
\end{cases}
\]
This distribution is henceforth denoted $\mathrm{APD}_{\lambda}(\bb{\theta})$; see Section~2.1 of \citet{MR4547729}.
\end{definition}

The $\mathrm{APD}_{\lambda}(1\hspace{-0.2mm}/2, \lambda, \mu, \sigma)$ distribution is referred to as the exponential power distribution, abbreviated $\mathrm{EPD}_{\lambda}(\mu, \sigma)$. For general location and scale parameters $\mu\in \R$ and $\sigma\in (0,\infty)$, its density is given, at any $x \in \R$, by
\[
g_{\lambda}(x \nvert \mu,\sigma)
= f_{\lambda}(x \nvert 1/2,\lambda,\mu,\sigma)
= \frac{1}{2\sigma \lambda^{1/\lambda} \Gamma(1 + 1/\lambda)} \exp\left\{-\frac{1}{\lambda} \, |\sigma^{-1}(x-\mu)|^{\lambda}\right\}.
\]
The special cases $\lambda = 1$ and $\lambda = 2$ correspond with the $\mathrm{Laplace}\hspace{0.3mm}(\mu, \sigma)$ and $\mathcal{N}(\mu, \sigma^2)$ distributions, respectively.

Given a continuous random sample $X_1, \ldots, X_n$, consider the omnibus test for the $\mathrm{EPD}_{\lambda}$ given by
\[
\mathcal{H}_0 : X_1, \ldots, X_n \sim \mathrm{EPD}_{\lambda}(\mu_0, \sigma_0) \equiv \mathrm{APD}_{\lambda}(\bb{\theta}_0) \quad \text{versus} \quad \mathcal{H}_1 : X_1, \ldots, X_n \centernot{\sim} \mathrm{EPD}_{\lambda}(\mu_0, \sigma_0),
\]
with $\bb{\theta}_0 = (1/2, \lambda, \mu_0, \sigma_0)^{\top}$. The purpose of viewing the $\mathrm{EPD}_{\lambda}$ as a special case of the $\mathrm{APD}_{\lambda}$ is to enable the selection of local alternatives within the $\mathrm{APD}_{\lambda}$ family. Following the notation in Section~\ref{sec:2}, the vector of parameters $\bb{\theta}$ has four dimensions ($p = 4$) and $F(\cdot \nvert \bb{\theta})$ denotes the cumulative distribution function of the $\mathrm{APD}_{\lambda}(\bb{\theta})$ distribution. The auxiliary parameter $\lambda$ is assumed to be known and fixed; it is excluded from the parameter vector $\bb{\theta}$.

As the parameters $\theta_{0,1} = 1/2$ and $\theta_{0,2} = \lambda$ are known, whereas the parameters $\mu_0$ and $\sigma_0$ are unknown, one sets
\[
\bb{\theta}_{\mathcal{K}} = (\theta_1,\theta_2)^{\top}, \quad \bb{\theta}_{0,\mathcal{K}} = (1/2,\lambda)^{\top}, \quad \bb{\theta}_{\mathcal{U}} = (\mu,\sigma)^{\top}, \quad \bb{\theta}_{0,\mathcal{U}} = (\mu_0,\sigma_0)^{\top},
\]
with $
\mathcal{K} = \{ 1, 2 \}$, $ \mathcal{U} = \mathcal{K}^{\complement} = \{3,4\}$, and $p_{\mathcal{K}} = p_{\mathcal{U}} = 2$.

For the local alternatives around the fixed parameter $\bb{\theta}_{0,\mathcal{K}} = (1/2,\lambda)^{\top}$, one has
\[
\mathcal{H}_{1,n}(\bb{\delta}_{\mathcal{K}}) : X_1, \ldots, X_n \sim \mathrm{APD}_{\lambda}(\bb{\theta}_n), \quad \bb{\theta}_n = \left[1/2 + \frac{\delta_1}{\sqrt{n}} \, \{1 + o(1)\}, \lambda + \frac{\delta_2}{\sqrt{n}} \, \{1 + o(1)\},\mu_0,\sigma_0 \right]^{\top},
\]
where the real vector $\bb{\delta}_{\mathcal{K}} = (\delta_1, \delta_2)^{\top}\in \R^2 \backslash \{\bb{0}_2\}$ is known and fixed.

The aforementioned modified score test for the $\mathrm{EPD}_{\lambda}(\mu_0, \sigma_0)$ distribution is based on Rao's score on the known parameters $\bb{\theta}_{\mathcal{K}} = (\theta_1, \theta_2)^{\top}$ and a replacement of the unknown parameters with consistent estimators. Both maximum likelihood estimators and method of moments estimators are considered in this section; see Proposition~\ref{prp:1} and Proposition~\ref{prp:2}, respectively. The test statistic is a bivariate $U$-statistic ($d = 2$) of degree $\nu = 1$ defined, for all $\bb{\theta} = (1/2, \lambda, \mu, \sigma)^{\top}$, by
\begin{equation}\label{eq:U.stat}
\begin{aligned}
\bb{U}_n(\bb{\theta})
&= \frac{1}{n} \sum_{i=1}^n \bb{h}(X_i \nvert \bb{\theta})
= \frac{1}{n} \sum_{i=1}^n \bb{s}_{\mathcal{K}}(X_i \nvert (1/2,\lambda,\mu,\sigma)^{\top})
= \frac{1}{n} \sum_{i=1}^n
\begin{bmatrix}
\left.\partial_{\theta_1} \ln \{ f_{\lambda}(X_i \nvert \bb{\theta}) \} \right|_{\bb{\theta} = (1/2,\lambda,\mu,\sigma)^{\top}} \\[1mm]
\left.\partial_{\theta_2} \ln \{ f_{\lambda}(X_i \nvert \bb{\theta}) \} \right|_{\bb{\theta} = (1/2,\lambda,\mu,\sigma)^{\top}}
\end{bmatrix} \\
&= \frac{1}{n} \sum_{i=1}^n
\begin{bmatrix}
-2 |Y_{i,\bb{\theta}_{\mathcal{U}}}|^{\lambda} \mathrm{sign}(Y_{i,\bb{\theta}_{\mathcal{U}}}) \\[1mm]
-\frac{1}{\lambda} \Big[|Y_{i,\bb{\theta}_{\mathcal{U}}}|^{\lambda} \ln |Y_{i,\bb{\theta}_{\mathcal{U}}}| - \frac{1}{\lambda} \left\{\ln (\lambda) + \psi(1+1/\lambda)\right\}\Big]
\end{bmatrix},
\end{aligned}
\end{equation}
with $Y_{i,\bb{\theta}_{\mathcal{U}}} = (X_i - \mu) / \sigma$, where the last equality in \eqref{eq:U.stat} is straightforward to verify using, e.g., \texttt{Mathematica}.

First, the maximum likelihood setting is considered. For a random sample $X_1,\ldots, X_n$ and for any fixed value of $\lambda\in [1,\infty)$, the maximum likelihood estimators of $\mu$ and $\sigma$ for the $\mathrm{APD}_{\lambda}(1/2, \lambda, \mu, \sigma) = \mathrm{EPD}_{\lambda}(\mu,\sigma)$ distribution are
\begin{equation}\label{eq:ML.estimators}
\hat{\mu}_n^{\star}
=
\begin{cases}
\mathrm{median}(X_1,\ldots,X_n), &\mbox{if } \lambda = 1, \\
(X_1 + \cdots + X_n)/n, &\mbox{if } \lambda = 2, \\
\mathrm{argmin}_{\mu\in \R} \sum_{i=1}^n |X_i - \mu|^{\lambda}, &\mbox{if } \lambda\in (1,\infty),
\end{cases}
\quad
\hat{\sigma}_n^{\star} = \left(\frac{1}{n} \sum_{i=1}^n |X_i - \hat{\mu}_n^{\star}|^{\lambda}\right)^{1/\lambda}.
\end{equation}
For details, see Proposition~2.3 of \citet{MR4547729}. In this setting, one has
\[
\hat{\bb{\theta}}_{n,\mathcal{K}} = (1/2, \lambda)^{\top}, \quad \hat{\bb{\theta}}_{n,\mathcal{U}} = (\hat{\mu}_n^{\star},\hat{\sigma}_n^{\star})^{\top}, \quad \hat{\bb{\theta}}_n = (1/2, \lambda,\hat{\mu}_n^{\star},\hat{\sigma}_n^{\star})^{\top}.
\]

The proposition below is a consequence of Theorems~\ref{thm:1} and \ref{thm:2}. For any $x \in (0,\infty)$, $\Gamma(x) = \int_0^{\infty} t^{x-1} e^{-t} \rd t$ denotes Euler's gamma function, $\psi(x) = \rd \ln \{\Gamma(x)\}/{\rd x}$ denotes the digamma function, and $\psi_1(x) = \rd \psi(x)/\rd x$ denotes the trigamma function.

\begin{prp}\label{prp:1}
Let $\lambda\in (1,\infty)$ and $\hat{\bb{\theta}}_n = (1/2,\lambda,\hat{\mu}_n^{\star},\hat{\sigma}_n^{\star})^{\top}$ using the maximum likelihood estimator $(\hat{\mu}_n^{\star},\hat{\sigma}_n^{\star})^{\top}$ in \eqref{eq:ML.estimators}, and consider $\bb{U}_n(\hat{\bb{\theta}}_n)$ using the $U$-statistic defined in \eqref{eq:U.stat}. Also, let
\[
\Sigma_{\bb{\theta}_0,\mathcal{U}} = \Sigma_{\hat{\bb{\theta}}_n, \mathcal{U}}
=
\begin{bmatrix}
4(1 + \lambda) - \frac{4\lambda}{\Gamma(2 - 1/\lambda) \Gamma(1+1/\lambda)} & 0 \\
0 & \lambda^{-3} C_{1,\lambda}
\end{bmatrix}, \quad
M_{\bb{\theta}_0} = \Sigma_{\bb{\theta}_0,\mathcal{U}},
\]
where $C_{1,\lambda} = (1+1/\lambda)\psi_1(1+1/\lambda)-1$. Under $\mathcal{H}_0$ and as $n \to \infty$, one has
\[
n \, \bb{U}_n(\hat{\bb{\theta}}_n)^{\top} \Sigma_{\bb{\theta}_0,\mathcal{U}}^{-1} \bb{U}_n(\hat{\bb{\theta}}_n) \rightsquigarrow \chi_2^2.
\]
Under $\mathcal{H}_{1,n}(\bb{\delta}_{\mathcal{K}})$ and as $n \to \infty$, one has
\[
n \, \bb{U}_n(\hat{\bb{\theta}}_n)^{\top}\Sigma_{\bb{\theta}_0, \mathcal{U}}^{-1} \bb{U}_n(\hat{\bb{\theta}}_n) \rightsquigarrow \chi_2^2\big(\bb{\delta}_{\mathcal{K}}^{\top} M_{\bb{\theta}_0}^{\top} \Sigma_{\bb{\theta}_0,\mathcal{U}}^{-1} \, M_{\bb{\theta}_0}\bb{\delta}_{\mathcal{K}}\big) = \chi_2^2\big(\delta_1^2 [\Sigma_{\bb{\theta}_0,\mathcal{U}}]_{11} + \delta_2^2 [\Sigma_{\bb{\theta}_0,\mathcal{U}}]_{22}\big).
\]
\end{prp}

Second, the method of moments setting is considered. Given that
\[
\EE(X) = \mu, \quad
\Var(X) = \frac{\lambda^{2/\lambda}\Gamma(1+3/\lambda)}{3\Gamma(1+1/\lambda)}\sigma^2,
\]
the joint method of moments estimators of $\mu$ and $\sigma$, matching the two first moments, are given by
\begin{equation}\label{eq:MOM.estimators}
\hat{\mu}_n = \frac{1}{n} \sum_{i=1}^n X_i, \quad
\hat{\sigma}_n = \left\{\frac{3\Gamma(1+1/\lambda)}{\lambda^{2/\lambda}\Gamma(1+3/\lambda)}\,\frac{1}{n} \sum_{i=1}^n (X_i - \hat{\mu}_n)^2\right\}^{1/2},
\end{equation}
respectively. In this setting, one has
\[
\hat{\bb{\theta}}_{n,\mathcal{K}} = (1/2,\lambda)^{\top}, \quad \hat{\bb{\theta}}_{n,\mathcal{U}} = (\hat{\mu}_n, \hat{\sigma}_n)^{\top}, \quad \hat{\bb{\theta}}_n = (1/2,\lambda,\hat{\mu}_n,\hat{\sigma}_n)^{\top}.
\]
The proposition below is a consequence of Theorems~\ref{thm:1}~and~\ref{thm:2}.

\begin{prp}\label{prp:2}
Let $\lambda\in (1,\infty)$ and $\hat{\bb{\theta}}_n = (1/2, \lambda, \hat{\mu}_n, \hat{\sigma}_n)^{\top}$ using the method of moments estimator $(\hat{\mu}_n, \hat{\sigma}_n)^{\top}$ in \eqref{eq:MOM.estimators}, and consider $\bb{U}_n (\hat{\bb{\theta}}_n)$ using the $U$-statistic defined in \eqref{eq:U.stat}. Also, let
\[
\begin{aligned}
\Sigma_{\bb{\theta}_0,\mathcal{U}}
&= \Sigma_{\hat{\bb{\theta}}_n, \mathcal{U}} =
\begin{bmatrix}
4(1 + \lambda) + \frac{4\lambda^2\Gamma(1+3/\lambda)}{3\Gamma^3(1+1/\lambda)}-\frac{8\lambda\Gamma(1+2/\lambda)}{\Gamma^2(1+1/\lambda)} & 0 \\[1mm]
0 & \lambda^{-3} \left(C_{1,\lambda}+2C_{2,\lambda}^2 + \frac{\lambda}{4} \, C_{2,\lambda}^2 C_{3,\lambda}^{-1}-C_{2,\lambda}C_{4,\lambda}\right)
\end{bmatrix}, \\
M_{\bb{\theta}_0}
&=
\begin{bmatrix}
4(1 + \lambda)-\frac{4\lambda\Gamma(1+2/\lambda)}{\Gamma^2(1+1/\lambda)} & 0 \\
0 & \lambda^{-3} \left(C_{1,\lambda} + \frac{3}{2}\, C_{2,\lambda}^2 - \frac{1}{2} \,C_{2,\lambda} C_{4,\lambda}\right)
\end{bmatrix},
\end{aligned}
\]
where
\[
\begin{aligned}
&C_{1,\lambda} = (1+1/\lambda)\psi_1(1+1/\lambda)-1, \quad
C_{2,\lambda} = 1+ \ln (\lambda) + \psi (1+1/\lambda), \\
&C_{3,\lambda} = \frac{\Gamma^2(1+3/\lambda)}{(9/5) \Gamma(1+1/\lambda) \Gamma(1+5/\lambda) - \Gamma^2(1+3/\lambda)}, \quad
C_{4,\lambda} = 1 + 3 \ln (\lambda) + 3 \psi (1+3/\lambda).
\end{aligned}
\]
Under $\mathcal{H}_0$ and as $n \to \infty$, one has
\[
n \, \bb{U}_n(\hat{\bb{\theta}}_n)^{\top} \Sigma_{\bb{\theta}_0,\mathcal{U}}^{-1} \bb{U}_n(\hat{\bb{\theta}}_n) \rightsquigarrow \chi_2^2.
\]
Under $\mathcal{H}_{1,n}(\bb{\delta}_{\mathcal{K}})$ and as $n \to \infty$, one has
\[
n \, \bb{U}_n(\hat{\bb{\theta}}_n)^{\top}\Sigma_{\bb{\theta}_0, \mathcal{U}}^{-1} \bb{U}_n(\hat{\bb{\theta}}_n) \rightsquigarrow \chi_2^2\big(\bb{\delta}_{\mathcal{K}}^{\top} M_{\bb{\theta}_0}^{\top} \Sigma_{\bb{\theta}_0,\mathcal{U}}^{-1} \, M_{\bb{\theta}_0}\bb{\delta}_{\mathcal{K}}\big) = \chi_2^2\big(\delta_1^2 [M_{\bb{\theta}_0}]_{11}^2 [\Sigma_{\bb{\theta}_0,\mathcal{U}}]_{11}^{-1} + \delta_2^2 [M_{\bb{\theta}_0}]_{22}^2 [\Sigma_{\bb{\theta}_0,\mathcal{U}}]_{22}^{-1}\big).
\]
\end{prp}

% FIGURE 1
\begin{figure}[b!]
\centering
\includegraphics[width=0.49\textwidth]{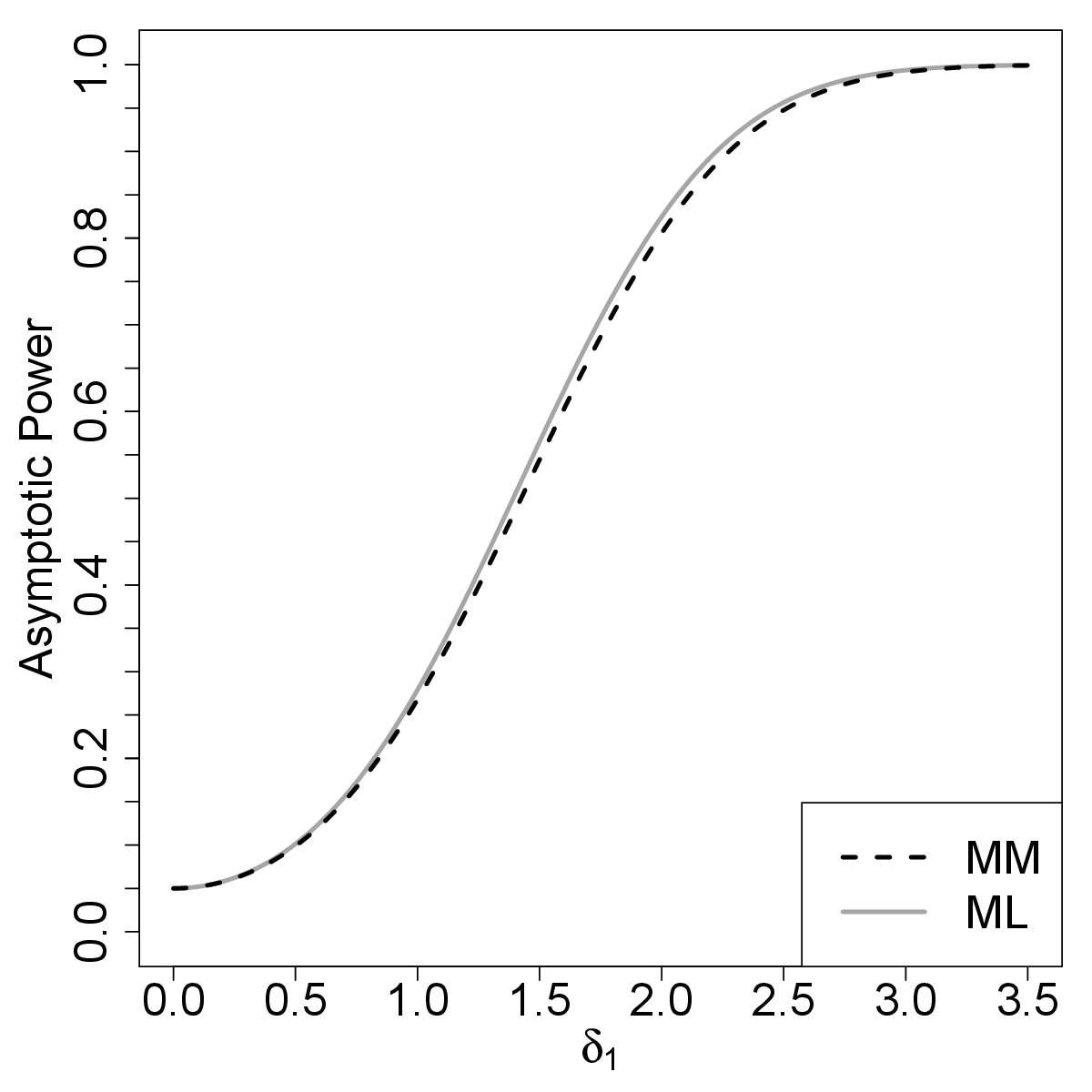}
\includegraphics[width=0.49\textwidth]{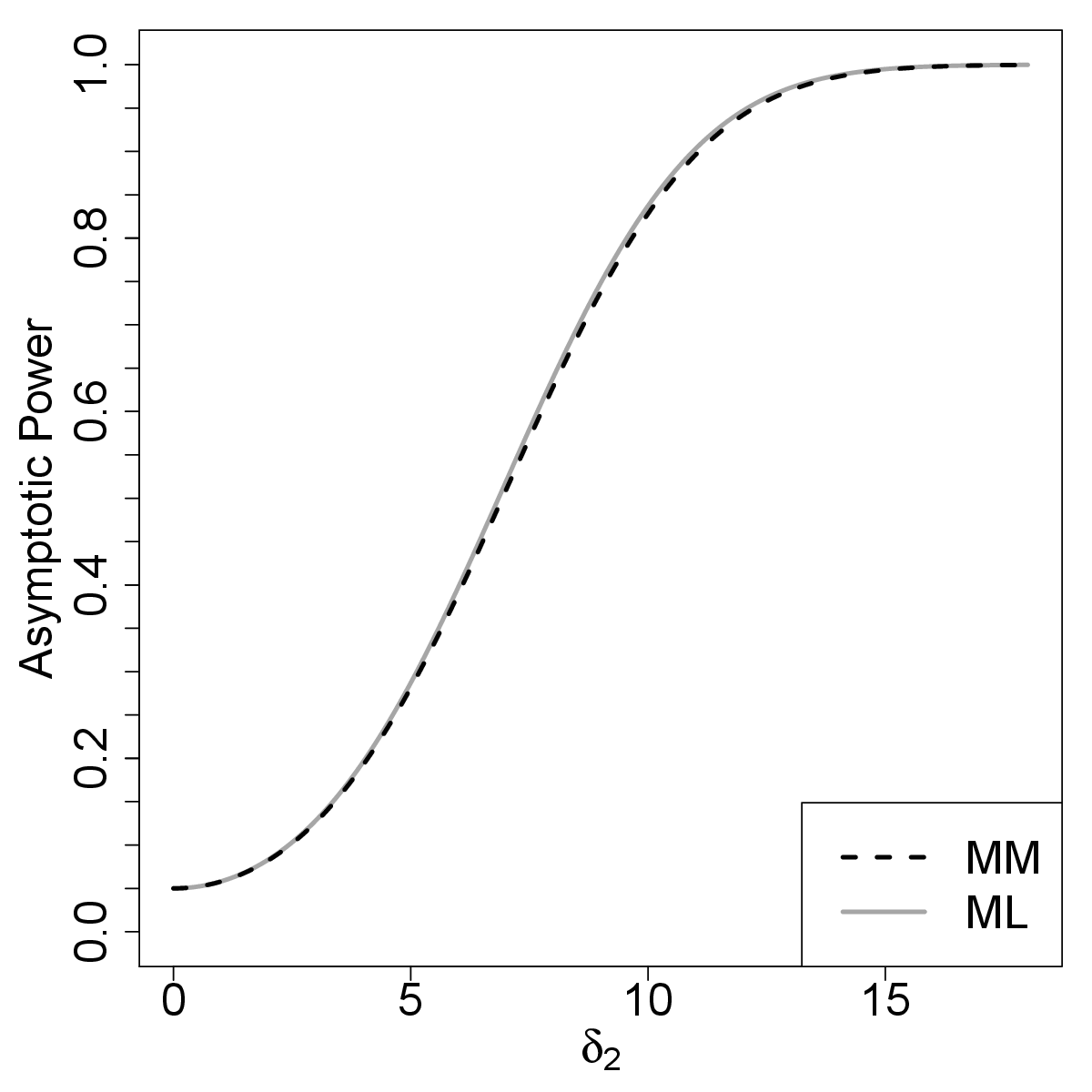}
\caption{Asymptotic power of the modified score goodness-of-fit tests for the exponential power distribution under local alternatives, using both maximum likelihood (ML) and the method of moments (MM) to estimate the unknown parameters $\mu_0$ and $\sigma_0$, whose specific values are irrelevant due to the location and scale invariance of the test statistic. The curves are displayed as a function of $\delta_1$ for a fixed $\delta_2 = 0$ (left) and as a function of $\delta_2$ for a fixed $\delta_1=0$ (right). The auxiliary parameter value is $\lambda=1.5$ and the nominal significance level is $0.05$.}\label{fig:1}
\end{figure}

\medskip
In Fig.~\ref{fig:1}, asymptotic power curves of the modified score test statistic $\smash{n \, \bb{U}_n(\hat{\bb{\theta}}_n)^{\top}\Sigma_{\bb{\theta}_0, \mathcal{U}}^{-1} \bb{U}_n(\hat{\bb{\theta}}_n)}$ for the exponential power distribution are shown under local alternatives for auxiliary parameter value $\lambda = 1.5$ and nominal significance level $\alpha = 5\%$. Given Propositions~\ref{prp:1}~and~\ref{prp:2}, the asymptotic critical value of the test is $\chi_{2,0.05}^2 \approx 5.991$ in each case.

In the left panel, the asymptotic power curves are compared under the local alternatives $\mathcal{H}_{1,n}(\delta_1, 0)$ for $\delta_1\in [0,3.5]$, and using maximum likelihood versus the method of moments to estimate the unknown parameters $\mu_0$ and $\sigma_0$, whose specific values are irrelevant due to the location and scale invariance of the test statistic. The power reaches its minimum at $\delta_1 = 0$, where it is equal to the significance level $0.05$, and gradually increases from 0.05 to 1 as $\delta_1$ moves away from 0, as expected.

In the right panel, the asymptotic power curves are compared under the local alternatives $\mathcal{H}_{1,n}(0, \delta_2)$ for $\delta_2\in [0,18]$, with similar results. The \textsf{R} code needed to reproduce this example and the graphs in Fig.~\ref{fig:1} is provided in the Supplement.

In Fig.~\ref{fig:2}, empirical power curves based on 10,000 Monte Carlo replications and various sample sizes are presented under the same setting as in Fig.~\ref{fig:1}, using maximum likelihood to estimate the unknown parameters $\mu_0$ and $\sigma_0$. The graphs based on the method of moments are omitted due to their similarity to the results shown.

Both the left and right panels demonstrate that the empirical curves converge to their asymptotic counterpart ($n = \infty$), regardless of whether $\delta_1$ or $\delta_2$ deviates from 0. However, in this example, the convergence is much faster when $\delta_1$ deviates from 0 (left panel) compared to when $\delta_2$ does (right panel). Not only are the sample sizes larger in the right panel, but they grow at twice the rate of those in the left panel. Nevertheless, the reader should keep in mind that these comments are specific to this example.

% FIGURE 2
\begin{figure}[t!]
\centering
\includegraphics[width=0.49\textwidth]{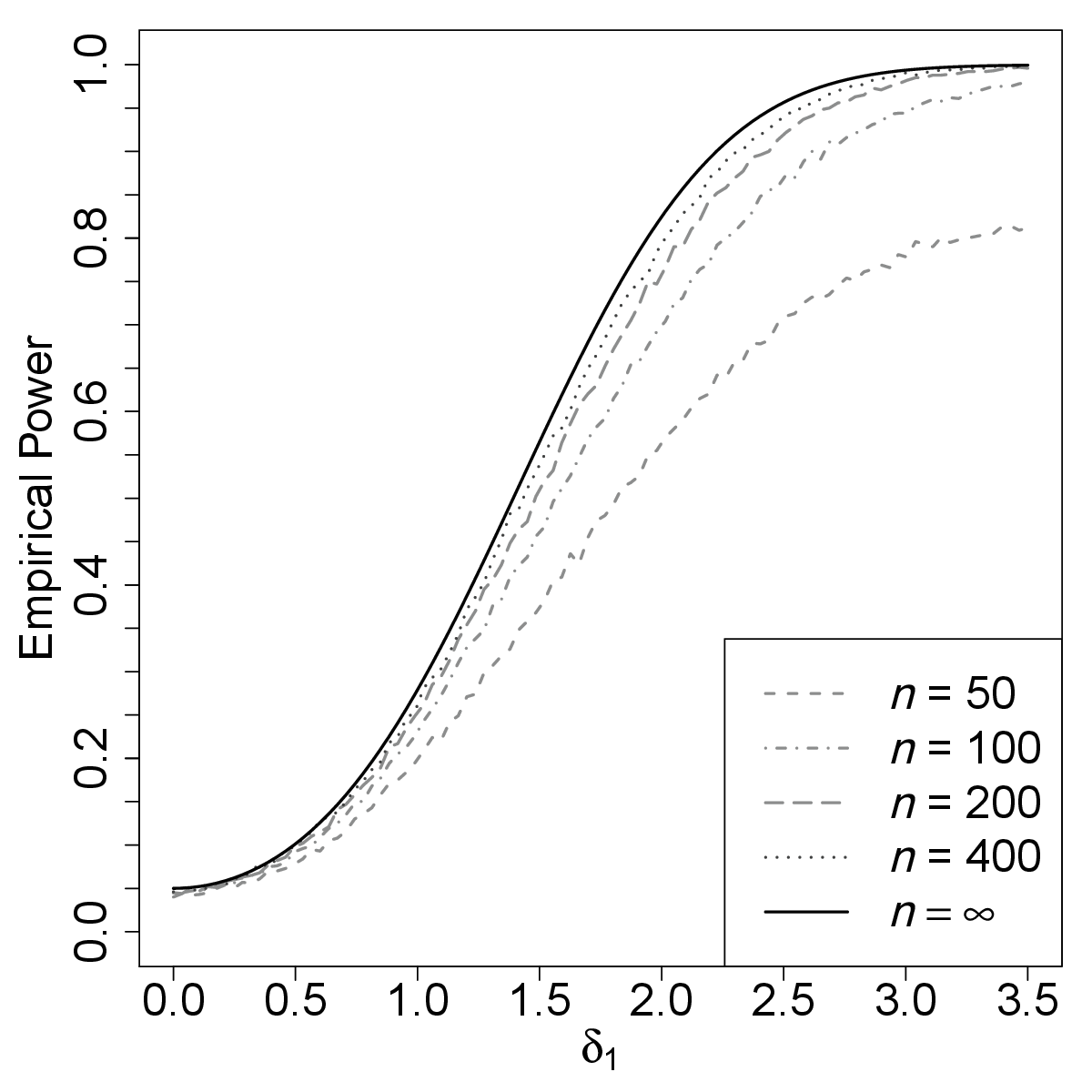}
\includegraphics[width=0.49\textwidth]{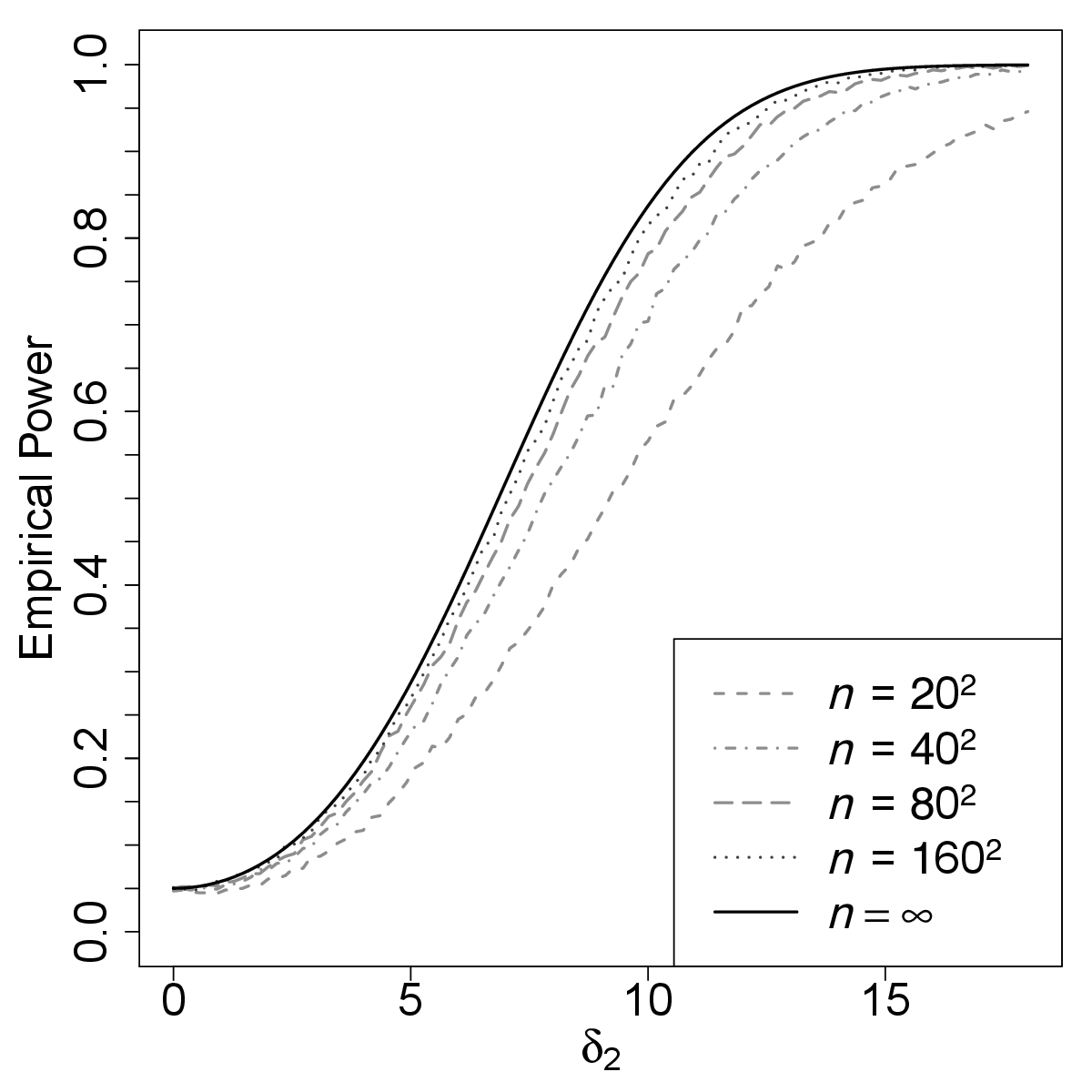}
\caption{Empirical power of the modified score goodness-of-fit tests for the exponential power distribution under local alternatives, using maximum likelihood to estimate the unknown parameters $\mu_0$ and $\sigma_0$. The curves are displayed as a function of $\delta_1$ for a fixed $\delta_2 = 0$ (left) and as a function of $\delta_2$ for a fixed $\delta_1=0$ (right). The auxiliary parameter value is $\lambda = 1.5$ and the nominal significance level is $0.05$. The case $n = \infty$ corresponds to the asymptotic power curve presented in Fig.~\ref{fig:1}.}\label{fig:2}
\end{figure}

\subsection{Example~2: Goodness-of-fit test for the normal distribution}
\label{sec:example.2}

Define the skew-normal density function, for all $\mu \in \R$, $\sigma^2 \in (0, \infty)$, $\alpha\in \R$, and $x\in \R$, by
\[
f(x \nvert \mu, \sigma^2, \alpha) = \frac{2}{\sigma} \, \phi \left( \frac{x - \mu}{\sigma} \right) \Phi \left( \alpha \, \frac{x - \mu}{\sigma} \right),
\]
where $\phi$ and $\Phi$ denote the density and cumulative distribution function of the standard normal distribution, respectively. This distribution was originally introduced by \citet{MR808153} and is also discussed in Chapter~12 of \citet{MR1299979}. If a random variable $X$ has this distribution, one writes $X\sim \mathrm{SN}(\mu,\sigma^2,\alpha)$ for short.

Given a continuous random sample $X_1, \dots, X_n$, consider testing
\[
\mathcal{H}_0 : X_1, \ldots, X_n \sim \mathrm{SN}(\mu_0,\sigma_0^2,0) \equiv \mathcal{N}(\mu_0,\sigma_0^2)
\]
versus
\[
\mathcal{H}_1 : X_1, \ldots, X_n \centernot{\sim} \mathrm{SN}(\mu_0,\sigma_0^2,0) ~\text{and the true distribution is asymmetric},
\]
with $\bb{\theta}_0 = (\mu_0, \sigma_0^2, \alpha_0)^{\top} \equiv (\mu_0, \sigma_0^2, 0)^{\top}$. The $U$-statistic considered in this section has degree 2. It is asymptotically equivalent to the Wilcoxon signed-rank statistic and  was studied, e.g., by \citet{MR1890341,MR2146496} and \citet{MR4103316} in the context of symmetry tests. It is defined, for all $\bb{\theta} = (\mu, \sigma^2, \alpha)^{\top}\in \R \times (0, \infty) \times \R$, by
\begin{equation}
\label{eq:example.2.U.stat}
U_n(\bb{\theta}) = \frac{2}{n(n-1)} \sum_{1\leq i < j\leq n} h(X_i, X_j \nvert \bb{\theta}),
\end{equation}
where the kernel $h$ is defined, for all $x_1, x_2\in \R$, by
\[
h(x_1, x_2 \nvert \bb{\theta}) = \ind_{\{x_1 + x_2 \geq 2 \mu\}} - 1/2.
\]

The parameter $\alpha_0 = 0$ is known in $\mathcal{H}_0$, and the parameters $\mu_0$ and $\sigma_0^2$ are assumed to be unknown, so one sets
\[
\bb{\theta} = (\mu,\sigma^2, \alpha)^{\top}, \quad \bb{\theta}_{\mathcal{K}} = \alpha, \quad \bb{\theta}_{\mathcal{U}} = (\mu, \sigma^2)^{\top}, \quad \bb{\theta}_0 = (\mu_0, \sigma_0^2, 0)^{\top}, \quad \bb{\theta}_{0,\mathcal{K}} = \alpha_0 = 0, \quad \bb{\theta}_{0,\mathcal{U}} = (\mu_0, \sigma_0^2)^{\top},
\]
with $\mathcal{K} = \{3\}$, $\mathcal{U} = \mathcal{K}^{\complement} = \{1,2\}$, $p_{\mathcal{K}} = 1$, and $p_{\mathcal{U}} = 2$.

To estimate the unknown parameters, the following maximum likelihood estimator is adopted:
\begin{equation}
\label{eq:example.2.MLE}
\hat{\bb{\theta}}_{n,\mathcal{U}} = (\hat{\mu}_n, \hat{\sigma}_n^2)^{\top}, \quad \hat{\mu}_n = \frac{1}{n} \sum_{i=1}^n x_i, \quad \hat{\sigma}_n^2 = \frac{1}{n} \sum_{i=1}^n (X_i - \hat{\mu}_n)^2.
\end{equation}

Given that the null distribution $\mathrm{SN}(\mu_0,\sigma_0^2,0)$ belongs to the larger parametric family $\{\mathrm{SN}(\mu,\sigma^2,\alpha) : \mu \in \R,$ $\sigma^2 \in (0,\infty), \alpha \in \R\}$, and $\bb{\theta}_{0,\mathcal{K}} = \alpha_0 = 0$ is assumed to be known under $\mathcal{H}_0$, the local alternatives are taken to be
\[
\mathcal{H}_{1,n}(\delta) : X_1, \ldots, X_n \sim \mathrm{SN}(\bb{\theta}_n), \quad \bb{\theta}_n = \left[\mu_0, \sigma_0^2, 0 + \frac{\delta}{\sqrt{n}} \{1 + o(1)\}\right]^{\top},
\]
where the real $\delta \equiv \delta_{\mathcal{K}}\in \R \backslash \{0\}$ is known and fixed.

The following proposition is a consequence of Theorems~\ref{thm:1}~and~\ref{thm:2}; see Fig.~\ref{fig:3} for an illustration.

% FIGURE 3
\begin{figure}[!b]
\centering
\includegraphics[width=0.95\textwidth]{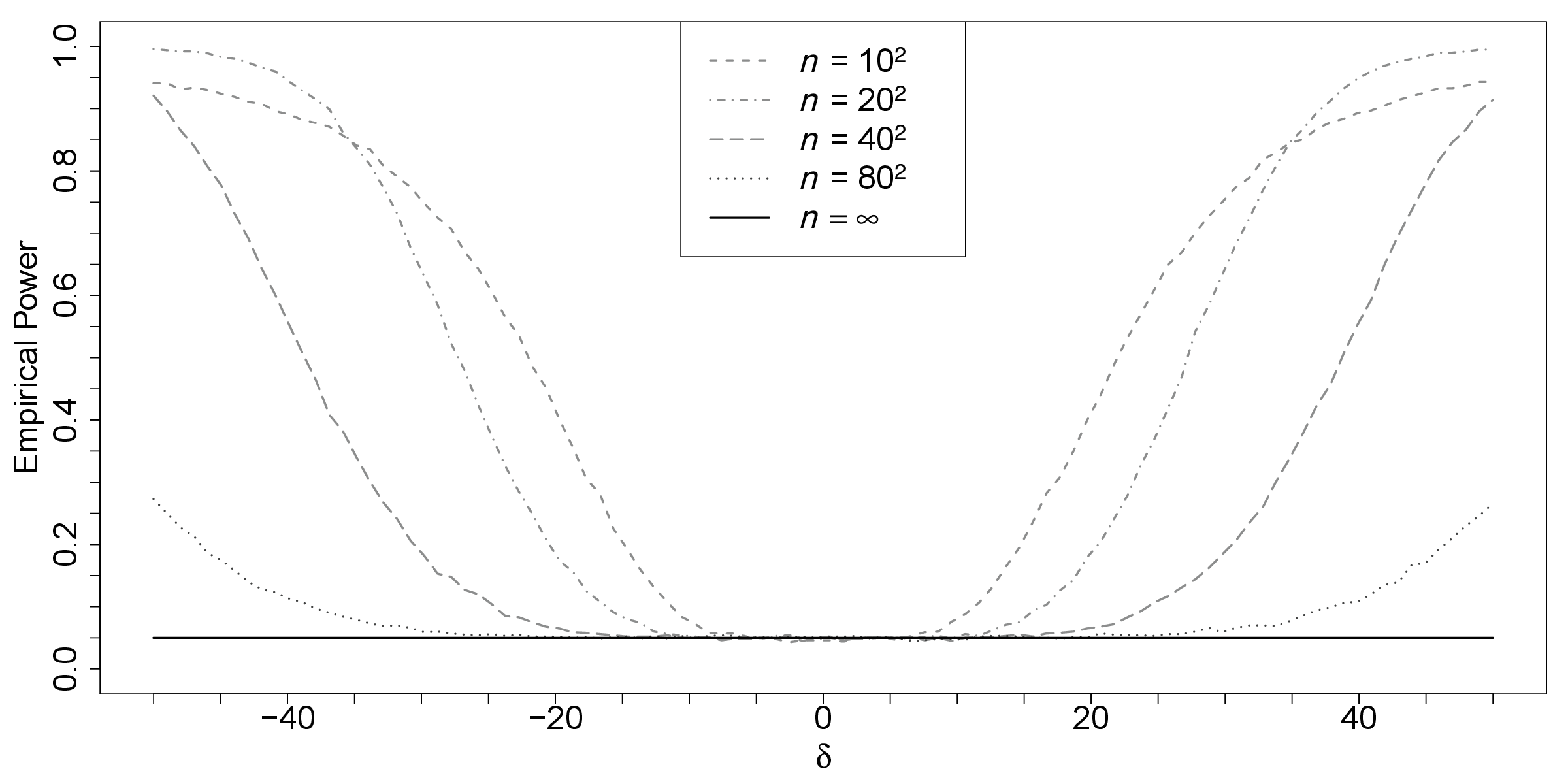}
\caption{Empirical power of $n \, U_n^2 (\hat{\bb{\theta}}_n) / (1/3 - 1/\pi)$ under local alternatives, using maximum likelihood to estimate the unknown parameter $\smash{(\mu_0, \sigma_0^2)^{\top}}$. The curves are displayed as a function of $\delta$. The nominal significance level is $0.05$. The case $n = \infty$ corresponds to the asymptotic power curve.}\label{fig:3}
\end{figure}

\begin{prp}\label{prp:3}
Let $\hat{\bb{\theta}}_n = (\hat{\mu}_n, \hat{\sigma}_n^2,0)^{\top}$, where $(\hat{\mu}_n, \hat{\sigma}_n^2)^{\top}$ is the maximum likelihood estimator in \eqref{eq:example.2.MLE}, and consider $\bb{U}_n(\hat{\bb{\theta}}_n)$ using the $U$-statistic defined in \eqref{eq:example.2.U.stat}. Also, let $\Sigma_{\bb{\theta}_0,\mathcal{U}} = 1/3 - 1/\pi$ and $M_{\bb{\theta}_0} = 0$. Under both $\mathcal{H}_0$ and $\mathcal{H}_{1,n}(\delta)$, one has, as $n \to \infty$,
\[
\frac{n \, U_n^2(\hat{\bb{\theta}}_n)}{1/3 - 1/\pi} \rightsquigarrow \chi_1^2.
\]
\end{prp}

Fig.~\ref{fig:3} supports Proposition~\ref{prp:3} by showing that for a fixed $\delta$ significantly smaller than $\sqrt{n}$, the empirical power curves align with the theoretical power curve, which remain constant at the nominal significance level of $0.05$, because the asymptotic distribution of $n \, U_n^2 (\hat{\bb{\theta}}_n) / (1/3 - 1/\pi)$ under the local alternatives $\mathcal{H}_{1,n}(\delta)$ is the (central) $\chi_1^2$ distribution. The empirical power curves diverge when $\delta = \delta' \sqrt{n}$ for some small positive $\delta'$. In contrast with Example~1, the asymptotic distribution of the normalized $U$-statistic is unaffected by perturbations of the known parameter (here, the asymmetry parameter $\alpha$ of the skew-normal distribution) in the order of $1/\sqrt{n}$. However, it is certainly influenced by perturbations of order $1$, as expected.

\section{Discussion}
\label{sec:6}

The results presented in this paper provide a comprehensive understanding of the asymptotic behavior of non-degenerate multivariate $U$-statistics with estimated nuisance parameters under both the null and local alternative hypotheses. However, several intriguing directions remain open for future research.

First, it would be a significant advancement to extend the results to $U$-statistics in the degenerate case, which are characterized by their Hoeffding decomposition having zero variance in the first term. The Hoeffding decomposition expresses a $U$-statistic as a sum of uncorrelated components, each representing projections onto different subspaces of symmetric functions. These projection terms are based on the degree of interaction they represent: the first term involves one variable, the second term involves two variables, and so on. Degenerate $U$-statistics are common in many applications, such as goodness-of-fit tests, tests for independence, symmetry tests, homogeneity tests, etc., where the test statistic's variance approaches zero as the sample size increases and its overall distributional behavior is controlled by higher-order interaction terms. As a result, standard asymptotic normality results become inapplicable. Instead, when properly normalized, the $U$-statistic converges to an infinite mixture of translated chi-square random variables, with the weights corresponding to the eigenvalues of the kernel's covariance operator \citep[pp.~79--80]{MR1075417}. Understanding these non-normal limiting distributions in this setting, for instance the impact of the estimated nuisance parameters on the eigenvalues, is crucial for accurate and reliable inference, and would provide more robust and applicable tools for statistical analysis in these common settings.

Second, considering $U$-processes seems like a promising avenue. Such processes generalize $U$-statistics by allowing the kernel to vary with an additional parameter, typically a time parameter. Analyzing $U$-processes in the presence of estimated nuisance parameters and under local alternatives could reveal new asymptotic properties and potential applications in statistical inference. For further reading on this topic (excluding estimated nuisance parameters), see, e.g., \cite{MR888439, MR942769, MR1235426, MR1239110, MR1348376, MR1256391, MR1272092, MR1666908, MR1311984, MR4716852, MR4562252, MR4366753, MR4674156, MR4695390}.

Third, it would be of interest to investigate the extension of the results herein to the settings of truncated, censored and/or missing data. Truncated data occur when observations falling outside a certain range are excluded from the analysis, such as income data where values below or above certain thresholds are not reported. Censored data are common in clinical studies where an event (like the onset of a disease) is known to have happened before or after a certain time but its occurrence time may be unknown. Missing data occur when some observations are not recorded or are lost, which frequently happens in survey responses. In many practical multivariate settings, it is common to face these issues, or a combination thereof, and to have variables that are unknown and need to be estimated to make inferences. Therefore, understanding the asymptotic behavior of $U$-statistics in this context is essential. For recent relevant publications in these areas (without estimated nuisance parameters), see, e.g., \cite{MR1276700,MR4366753,MR2779643,MR4623301,MR3911661,MR4619078}.

Finally, another avenue mentioned by a referee would be the study of Bahadur efficiency in this setting. Bahadur efficiency evaluates the performance of a statistical test based on the rate at which the $p$-value of the test statistic converges to zero under the alternative hypothesis and is rooted in large deviation theory; see, e.g., \citep[p.~203]{MR1652247}. Understanding Bahadur efficiency in the context of $U$-statistics with estimated nuisance parameters and under local alternatives would be a significant contribution to the theory of statistical inference. However, it is possible that a different toolset would be required to obtain results in this area, as it lies within the realm of large deviation theory, in contrast to the stochastic Taylor expansions, weak limit theorems, and contiguity methods used in the current work.

In conclusion, while this paper addresses the asymptotic behavior of multivariate $U$-statistics with estimated nuisance parameters, exploring these additional directions would greatly enrich the field and open up new possibilities for both theoretical advancements and practical applications. Everything mentioned above could also be reframed in the context of $V$-statistics, where most of the same questions remain open.

\section{Proofs of the main results}\label{sec:proofs}

\begin{proof}[\bf Proof of Theorem~\ref{thm:1}]
The first step of the proof is to expand $\sqrt{n} \, \bb{U}_n(\hat{\bb{\theta}}_n)$ using a stochastic Taylor expansion of order~$1$; see, e.g., Theorem~18.18~of~\citet{MR2378491}. Recalling that $\hat{\bb{\theta}}_{n,\mathcal{K}} = \bb{\theta}_{0,\mathcal{K}}$, there exists a random variable $Y$ taking values in $[0, 1]$ such that
\[
\bb{U}_n(\hat{\bb{\theta}}_n) = \bb{U}_n(\bb{\theta}_0) + \bb{U}_{n,\mathcal{U}}'\{\bb{\theta}_0 + Y (\hat{\bb{\theta}}_n - \bb{\theta}_0)\} (\hat{\bb{\theta}}_{n,\mathcal{U}} - \bb{\theta}_{0,\mathcal{U}}),
\]
where $\bb{U}_n'(\bb{\theta}) = \partial_{\bb{\theta}^{\top}}\bb{U}_n(\bb{\theta}) = \big( \partial_{\theta_1} \bb{U}_n(\bb{\theta}), \ldots, \partial_{\theta_p} \bb{U}_n(\bb{\theta})\big)$ and $\bb{U}_{n,\mathcal{U}}'(\bb{\theta}) = \partial_{\bb{\theta}_{\mathcal{U}}^{\top}} \bb{U}_n(\bb{\theta})$. By adding and subtracting both $\bb{U}_{n,\mathcal{U}}'(\bb{\theta}_0) (\hat{\bb{\theta}}_{n,\mathcal{U}} - \bb{\theta}_{0,\mathcal{U}})$ and $ \EE\{\bb{U}_{n,\mathcal{U}}'(\bb{\theta}_0)\} (\hat{\bb{\theta}}_{n,\mathcal{U}} - \bb{\theta}_{0,\mathcal{U}})$, one has
\begin{align*}
\bb{U}_n(\hat{\bb{\theta}}_n)
= \bb{U}_n(\bb{\theta}_0)
&+ \bb{U}_{n,\mathcal{U}}'(\bb{\theta}_0) (\hat{\bb{\theta}}_{n,\mathcal{U}} - \bb{\theta}_{0,\mathcal{U}}) \\
&+ \left[\bb{U}_{n,\mathcal{U}}'\{\bb{\theta}_0 + Y (\hat{\bb{\theta}}_n - \bb{\theta}_0)\} - \EE\{\bb{U}_{n,\mathcal{U}}'(\bb{\theta}_0)\}\right] (\hat{\bb{\theta}}_{n,\mathcal{U}} - \bb{\theta}_{0,\mathcal{U}}) - \left[\bb{U}_{n,\mathcal{U}}'(\bb{\theta}_0) - \EE\{\bb{U}_{n,\mathcal{U}}'(\bb{\theta}_0)\}\right] (\hat{\bb{\theta}}_{n,\mathcal{U}} - \bb{\theta}_{0,\mathcal{U}}).
\end{align*}
Therefore,
\begin{equation}\label{eq:bound.after.Taylor}
\left|\bb{U}_n(\hat{\bb{\theta}}_n) - \left\{\bb{U}_n(\bb{\theta}_0) + \bb{U}_{n,\mathcal{U}}'(\bb{\theta}_0) (\hat{\bb{\theta}}_{n,\mathcal{U}} - \bb{\theta}_{0,\mathcal{U}})\right\}\right| \\
\leq 2 \sup_{y\in [0, 1]} \left|\bb{U}_{n,\mathcal{U}}'\{\bb{\theta}_0 + y (\hat{\bb{\theta}}_n - \bb{\theta}_0)\} - \EE\{\bb{U}_{n,\mathcal{U}}'(\bb{\theta}_0)\}\right| (\hat{\bb{\theta}}_{n,\mathcal{U}} - \bb{\theta}_{0,\mathcal{U}}).
\end{equation}

Under items $(a)$ and $(b)$ of Assumption~\ref{ass:3}, and using the fact that $\hat{\bb{\theta}}_n \stackrel{\PP_{\!\mathcal{H}_0}}{\longrightarrow} \bb{\theta}_0$ by Assumption~\ref{ass:4}, an application of the uniform law of large numbers in Lemma~\ref{lem:1} of \ref{app:A} to each of the $d \times p_{\mathcal{U}}$ components of $\bb{U}_{n,\mathcal{U}}'$ in \eqref{eq:bound.after.Taylor} with the choice
\[
g(\bb{x}_1, \ldots, \bb{x}_{\nu} \nvert \bb{\theta}) = \partial_{\theta_j} h_i(\bb{x}_1, \ldots, \bb{x}_{\nu} \nvert \bb{\theta}), \quad (i, j)\in \{1, \ldots, d\} \times \mathcal{U},
\]
proves that the right-hand side of \eqref{eq:bound.after.Taylor} is $o_{\hspace{0.3mm}\PP_{\!\mathcal{H}_0}}(1) \bb{1}_d$. Thus, one can write, as $n \to \infty$,
\begin{equation}
\label{eq:expansion}
\sqrt{n} \, \bb{U}_n(\hat{\bb{\theta}}_n) = \sqrt{n} \, \bb{U}_n(\bb{\theta}_0) + \bb{U}_{n,\mathcal{U}}'(\bb{\theta}_0) \sqrt{n} \, (\hat{\bb{\theta}}_{n,\mathcal{U}} - \bb{\theta}_{0,\mathcal{U}}) + o_{\hspace{0.3mm}\PP_{\!\mathcal{H}_0}}(1) \bb{1}_d.
\end{equation}

The asymptotic expression for $\sqrt{n} \, (\hat{\bb{\theta}}_{n,\mathcal{U}} - \bb{\theta}_{0,\mathcal{U}})$ is known by Assumption~\ref{ass:4}. Hence, the second step consists in studying the asymptotics of the term $\bb{U}_{n,\mathcal{U}}'(\bb{\theta}_0)$. Under Assumption~\ref{ass:3}, one can swap derivative and integral in the following equation:
\begin{multline*}
\partial_{\bb{\theta}_{\mathcal{U}}^{\top}} \int_{\R^m} \cdots \int_{\R^m} \bb{h}(\bb{x}_1, \ldots, \bb{x}_{\nu} \nvert \bb{\theta}_0) \textstyle \prod_{\ell=1}^{\nu} f(\bb{x}_{\ell} \nvert \bb{\theta}_0) \rd \bb{x}_1 \cdots \rd \bb{x}_{\nu} \\
= \int_{\R^m} \cdots \int_{\R^m} \partial_{\bb{\theta}_{\mathcal{U}}^{\top}} \big\{\bb{h}(\bb{x}_1, \ldots, \bb{x}_{\nu} \nvert \bb{\theta}_0) \textstyle\prod_{\ell=1}^{\nu} f(\bb{x}_{\ell} \nvert \bb{\theta}_0)\big\} \rd \bb{x}_1 \cdots \rd \bb{x}_{\nu} \\
\hspace{3.5cm}= \int_{\R^m} \cdots \int_{\R^m} \big\{\partial_{\bb{\theta}_{\mathcal{U}}^{\top}}\bb{h}(\bb{x}_1, \ldots, \bb{x}_{\nu} \nvert \bb{\theta}_0)\big\} \textstyle\prod_{\ell=1}^{\nu} f(\bb{x}_{\ell} \nvert \bb{\theta}_0) \rd \bb{x}_1 \cdots \rd \bb{x}_{\nu} \\
+ \int_{\R^m} \cdots \int_{\R^m} \bb{h}(\bb{x}_1, \ldots, \bb{x}_{\nu} \nvert \bb{\theta}_0) \big\{\partial_{\bb{\theta}_{\mathcal{U}}} \textstyle\prod_{\ell=1}^{\nu} f(\bb{x}_{\ell} \nvert \bb{\theta}_0)\big\}^{\top} \rd \bb{x}_1 \cdots \rd \bb{x}_{\nu}.
\end{multline*}

Given that the left-hand side of the above equation is $\bb{0}_{d \times p_{\mathcal{U}}}$ by Assumption~\ref{ass:1}, one has
\[
\bb{0}_{d \times p_{\mathcal{U}}}
= \EE\big\{\partial_{\bb{\theta}_{\mathcal{U}}^{\top}}\bb{h}(\bb{X}_1, \ldots, \bb{X}_{\nu} \nvert \bb{\theta}_0)\big\} + \nu \EE\big[\bb{h}(\bb{X}_1, \ldots, \bb{X}_{\nu} \nvert \bb{\theta}_0) \partial_{\bb{\theta}_{\mathcal{U}}^{\top}} \ln \{ f(\bb{X} \nvert\bb{\theta}_0) \}\big]
\equiv \EE\big\{\partial_{\bb{\theta}_{\mathcal{U}}^{\top}}\bb{h}(\bb{X}_1, \ldots, \bb{X}_{\nu} \nvert \bb{\theta}_0)\big\} + \nu G_{\bb{\theta}_0,{\mathcal{U}}}.
\]
Therefore, using the weak law of large numbers, one deduces that
\begin{equation}\label{eq:U.derivative}
\bb{U}_{n,\mathcal{U}}'(\bb{\theta}_0)
= \binom{n}{\nu}^{-1} \sum_{(n, \nu)} \partial_{\bb{\theta}_{\mathcal{U}}^{\top}}\bb{h}(\bb{X}_{i_1}, \ldots, \bb{X}_{i_{\nu}} \nvert \bb{\theta}_0)
= - \nu G_{\bb{\theta}_0,{\mathcal{U}}} + o_{\hspace{0.3mm}\PP_{\!\mathcal{H}_0}}(1) \bb{1}_d^{\phantom{\top}}\hspace{-1mm}\bb{1}_{p_{\mathcal{U}}}^{\top}.
\end{equation}

As a final step, one applies the expansions found in \eqref{eq:U.derivative} and Assumption~\ref{ass:4} into \eqref{eq:expansion} to obtain
\begin{equation}
\label{eq:thm.1.final.step}
\sqrt{n} \, \bb{U}_n(\hat{\bb{\theta}}_n)
=
\begin{bmatrix}
\mathrm{Id}_{d\times d} \, ; \, - \nu G_{\bb{\theta}_0,\mathcal{U}}R_{\bb{\theta}_0,\mathcal{U}}^{-1}
\end{bmatrix}
\begin{bmatrix}
\sqrt{n} \, \bb{U}_n(\bb{\theta}_0) \\[1mm]
\frac{1}{\sqrt{n}} \sum_{i=1}^n \bb{r}_{\mathcal{U}}(\bb{X}_i \nvert \bb{\theta}_0)
\end{bmatrix}
+ o_{\hspace{0.3mm}\PP_{\!\mathcal{H}_0}}(1)\bb{1}_d.
\end{equation}

By the central limit theorem and Assumptions~\ref{ass:2}~and~\ref{ass:4}, one has $\sqrt{n} \, \bb{U}_n(\hat{\bb{\theta}}_n) \rightsquigarrow \mathcal{N}_d(\bb{0}_d, \Sigma_{\bb{\theta}_0,\mathcal{U}})$ under $\mathcal{H}_0$ as $n\to \infty$, with
\[
\begin{aligned}
\Sigma_{\bb{\theta}_0,\mathcal{U}}
&=
\begin{bmatrix}
\mathrm{Id}_{d\times d} \, ; \, - \nu G_{\bb{\theta}_0,\mathcal{U}}R_{\bb{\theta}_0,\mathcal{U}}^{-1}
\end{bmatrix}
\begin{bmatrix}
\nu^2 H_{\bb{\theta}_0} & \nu J_{\bb{\theta}_0,\mathcal{U}} \\[1mm]
\nu J_{\bb{\theta}_0,\mathcal{U}}^{\top} & R_{\bb{\theta}_0,\mathcal{U}}
\end{bmatrix}
\begin{bmatrix}
\mathrm{Id}_{d\times d} \\[1mm]
- \nu R_{\bb{\theta}_0,\mathcal{U}}^{-1} G_{\bb{\theta}_0,\mathcal{U}}^{\top}
\end{bmatrix} \\
&=
\nu^2 \big\{H_{\bb{\theta}_0} - G_{\bb{\theta}_0,\mathcal{U}} R_{\bb{\theta}_0,\mathcal{U}}^{-1} J_{\bb{\theta}_0,\mathcal{U}}^{\top} - J_{\bb{\theta}_0,\mathcal{U}} R_{\bb{\theta}_0,\mathcal{U}}^{-1} G_{\bb{\theta}_0,\mathcal{U}}^{\top} + G_{\bb{\theta}_0,\mathcal{U}} R_{\bb{\theta}_0,\mathcal{U}}^{-1} G_{\bb{\theta}_0,\mathcal{U}}^{\top}\big\},
\end{aligned}
\]
given that $ \EE\{\bb{h}(\bb{X}_{i_1}, \ldots, \bb{X}_{i_{\nu}} \nvert \bb{\theta}_0)\} = \bb{0}_d$ and $ \EE\{\bb{r}_{\mathcal{U}}(\bb{X}_i \nvert \bb{\theta}_0)\} = \bb{0}_{p_{\mathcal{U}}}$ by Assumptions~\ref{ass:1}~and~\ref{ass:4}. This proves the first assertion.

Under the assumption that the matrix function $\Sigma_{\bb{\theta},\mathcal{U}}$ is almost-everywhere continuous in $\bb{\theta}$, and given that $\hat{\bb{\theta}}_n$ converges in $\PP_{\!\mathcal{H}_0}$-probability to $\bb{\theta}_0$ by Assumption~\ref{ass:4}, it follows from the continuous mapping theorem that $\Sigma_{\hat{\bb{\theta}}_n,\mathcal{U}}$ converges in $\PP_{\!\mathcal{H}_0}$-probability to $\Sigma_{\bb{\theta}_0,\mathcal{U}}$ component-wise. The second assertion then follows from the first result and Slutsky's lemma.
\end{proof}

\begin{proof}[\bf Proof of Theorem~\ref{thm:2}]
As pointed in the paragraph preceding Lemma~7.6 of \citet{MR1652247}, item (c) of Assumption~\ref{ass:3} implies that the map $\bb{\theta}\mapsto f(\cdot \nvert \bb{\theta})$ is differentiable in quadratic mean at the interior point $\bb{\theta}_0\in \Theta\subseteq \R^p$. Under this assumption, and given that $\bb{\theta}_{n,\mathcal{U}} = \bb{\theta}_{0,\mathcal{U}}$ under $\mathcal{H}_{1,n}(\bb{\delta}_{\mathcal{K}})$, Theorem~7.2 of \citet{MR1652247} implies that, as $n \to \infty$,
\begin{equation}\label{eq:log.ratio}
\ln ( {\rd \PP_{\!\mathcal{H}_{1,n}}}/{\rd \PP_{\!\mathcal{H}_0}} ) = \frac{\bb{\delta}^{\top}_{\mathcal{K}}}{\sqrt{n}} \sum_{i=1}^n \bb{s}_{\mathcal{K}}(\bb{X}_i \nvert \bb{\theta}_0) - \frac{1}{2} \, \bb{\delta}^{\top}_{\mathcal{K}} I_{\bb{\theta}_0,\mathcal{K}} \bb{\delta}_{\mathcal{K}} + o_{\hspace{0.3mm}\PP_{\!\mathcal{H}_0}}(1).
\end{equation}

From this fact, it is straightforward to check that items (ii) and (iii) in Le Cam's first lemma (see, e.g., Lemma~6.4 of \citet{MR1652247}) are verified, which implies that the sequences $(\PP_{\!\mathcal{H}_0})_{n\in \N}$ and $(\PP_{\!\mathcal{H}_{1,n}})_{n\in \N}$ of probability measures are mutually contiguous. By (iv) of the same lemma, this mutual contiguity is equivalent to the property that, for any random vector $\bb{T}_{\!n} = \bb{T}_{\!n}(\bb{X}_1, \ldots, \bb{X}_n \nvert \bb{\theta}_0)$ taking values in $\R^d$,
\begin{equation}\label{eq:contiguity}
\bb{T}_{\!n} \xrightarrow{\PP_{\!\mathcal{H}_0}} 0 \quad \Longleftrightarrow \quad \bb{T}_{\!n} \xrightarrow{\PP_{\!\mathcal{H}_{1,n}}} 0,
\end{equation}
where in both cases, the convergence occurs as $n \to \infty$. Now, the first assertion of Theorem~\ref{thm:2} can be tackled. Using the above expansion in \eqref{eq:log.ratio}, the central limit theorem together with Assumptions~\ref{ass:2}~and~\ref{ass:4} yield, as $n \to \infty$,
\begin{align*}
\begin{bmatrix}
\sqrt{n} \, \bb{U}_n(\bb{\theta}_0) \\[1mm]
\frac{1}{\sqrt{n}} \sum_{i=1}^n \bb{r}_{\mathcal{U}}(\bb{X}_i \nvert \bb{\theta}_0) \\[1mm]
\ln ( {\rd \PP_{\!\mathcal{H}_{1,n}}}/{\rd \PP_{\!\mathcal{H}_0}} )
\end{bmatrix}
&=
\begin{bmatrix}
\bb{0}_d \\[1mm]
\bb{0}_{p_{\mathcal{U}}} \\[1mm]
-\frac{1}{2} \, \bb{\delta}^{\top}_{\mathcal{K}} I_{\bb{\theta}_0,\mathcal{K}} \bb{\delta}_{\mathcal{K}} + o_{\hspace{0.3mm}\PP_{\!\mathcal{H}_0}}(1)
\end{bmatrix}
+
\begin{bmatrix}
\sqrt{n} \, \bb{U}_n(\bb{\theta}_0) \\[1mm]
\frac{1}{\sqrt{n}} \sum_{i=1}^n \bb{r}_{\mathcal{U}}(\bb{X}_i \nvert \bb{\theta}_0) \\[1mm]
\frac{\bb{\delta}_{\mathcal{K}}^{\top}}{\sqrt{n}} \sum_{i=1}^n \bb{s}_{\mathcal{K}}(\bb{X}_i \nvert \bb{\theta}_0)
\end{bmatrix} \\
&\stackrel{\PP_{\!\mathcal{H}_0}}{\scalebox{2}[1.2]{$\rightsquigarrow$}}
\mathcal{N}_{d + p_{\mathcal{U}} + 1}\left(
\begin{bmatrix}
\bb{0}_d \\[1mm]
\bb{0}_{p_{\mathcal{U}}} \\[1mm]
-\frac{1}{2} \, \bb{\delta}^{\top}_{\mathcal{K}} I_{\bb{\theta}_0,\mathcal{K}} \bb{\delta}_{\mathcal{K}}
\end{bmatrix},
\begin{bmatrix}
\nu^2 H_{\bb{\theta}_0} & \nu J_{\bb{\theta}_0,\mathcal{U}} & \nu G_{\bb{\theta}_0,\mathcal{K}} \bb{\delta}_{\mathcal{K}} \\[1mm]
\nu J_{\bb{\theta}_0,\mathcal{U}}^{\top} & R_{\bb{\theta}_0,\mathcal{U}} & S_{\!\bb{\theta}_0,\mathcal{K}, \mathcal{U}}^{\top} \bb{\delta}_{\mathcal{K}} \\[1mm]
\nu \bb{\delta}_{\mathcal{K}}^{\top} G_{\bb{\theta}_0,\mathcal{K}}^{\top} & \bb{\delta}_{\mathcal{K}}^{\top} S_{\!\bb{\theta}_0,\mathcal{K}, \mathcal{U}} & \bb{\delta}^{\top}_{\mathcal{K}} I_{\bb{\theta}_0,\mathcal{K}} \bb{\delta}_{\mathcal{K}}
\end{bmatrix}
\right).
\end{align*}
Then, by Le Cam's third lemma (see, e.g., Lemma~6.7 of \citet{MR1652247}), one finds that, as $n \to \infty$,
\[
\begin{bmatrix}
\sqrt{n} \, \bb{U}_n(\bb{\theta}_0) \\[1mm]
\frac{1}{\sqrt{n}} \sum_{i=1}^n \bb{r}_{\mathcal{U}}(\bb{X}_i \nvert \bb{\theta}_0)
\end{bmatrix}
\stackrel{\PP_{\!\mathcal{H}_{1,n}}}{\scalebox{2}[1.2]{$\rightsquigarrow$}}
\mathcal{N}_{d + p_{\mathcal{U}}}\left(
\begin{bmatrix}
\nu G_{\bb{\theta}_0,\mathcal{K}} \bb{\delta}_{\mathcal{K}} \\[1mm]
S_{\!\bb{\theta}_0,\mathcal{K}, \mathcal{U}}^{\top}\bb{\delta}_{\mathcal{K}}
\end{bmatrix},
\begin{bmatrix}
\nu^2 H_{\bb{\theta}_0} & \nu J_{\bb{\theta}_0,\mathcal{U}} \\[1mm]
\nu J_{\bb{\theta}_0,\mathcal{U}}^{\top} & R_{\bb{\theta}_0,\mathcal{U}}
\end{bmatrix}
\right).
\]
By applying the contiguity result \eqref{eq:contiguity} to \eqref{eq:thm.1.final.step}, one deduces that, as $n \to \infty$,
\[
\sqrt{n} \, \bb{U}_n(\hat{\bb{\theta}}_n)
=
\begin{bmatrix}
\mathrm{Id}_{d\times d} \, ; \, - \nu G_{\bb{\theta}_0,\mathcal{U}} R_{\bb{\theta}_0,\mathcal{U}}^{-1}
\end{bmatrix}
\begin{bmatrix}
\sqrt{n} \, \bb{U}_n(\bb{\theta}_0) \\[1mm]
\frac{1}{\sqrt{n}} \sum_{i=1}^n \bb{r}_{\mathcal{U}}(\bb{X}_i \nvert \bb{\theta}_0)
\end{bmatrix}
+ o_{\hspace{0.3mm}\PP_{\!\mathcal{H}_{1,n}}}(1)\bb{1}_d
\stackrel{\PP_{\!\mathcal{H}_{1,n}}}{\scalebox{2}[1.2]{$\rightsquigarrow$}}
\mathcal{N}_d(M_{\bb{\theta}_0}\bb{\delta}_{\mathcal{K}}, \Sigma_{\bb{\theta}_0,\mathcal{U}}),
\]
where the matrix $\Sigma_{\bb{\theta}_0,\mathcal{U}}$ has the same expression as in Theorem~\ref{thm:1}, and
\[
M_{\bb{\theta}_0}\bb{\delta}_{\mathcal{K}}
=
\begin{bmatrix}
\mathrm{Id}_{d\times d} \, ; \, - \nu G_{\bb{\theta}_0,\mathcal{U}} R_{\bb{\theta}_0,\mathcal{U}}^{-1}
\end{bmatrix}
\begin{bmatrix}
\nu G_{\bb{\theta}_0,\mathcal{K}} \bb{\delta}_{\mathcal{K}} \\[1mm]
S_{\!\bb{\theta}_0,\mathcal{K}, \mathcal{U}}^{\top}\bb{\delta}_{\mathcal{K}}
\end{bmatrix}
=
\nu \big\{G_{\bb{\theta}_0,\mathcal{K}} - G_{\bb{\theta}_0,\mathcal{U}} R_{\bb{\theta}_0,\mathcal{U}}^{-1} S_{\!\bb{\theta}_0,\mathcal{K}, \mathcal{U}}^{\top}\big\} \bb{\delta}_{\mathcal{K}}.
\]
This proves the first assertion.

Under the assumption that the matrix function $\Sigma_{\bb{\theta},\mathcal{U}}$ is almost-everywhere continuous in $\bb{\theta}$, and given that $\hat{\bb{\theta}}_n$ converges in $\PP_{\!\mathcal{H}_{1,n}}$-probability to $\bb{\theta}_0$ by Assumption~\ref{ass:4} and the contiguity result in \eqref{eq:contiguity}, the continuous mapping theorem shows that $\Sigma_{\hat{\bb{\theta}}_n,\mathcal{U}}$ converges in $\PP_{\!\mathcal{H}_{1,n}}$-probability to $\Sigma_{\bb{\theta}_0,\mathcal{U}}$ component-wise. Then, the second assertion of the theorem follows from the first one in conjunction with Slutsky's theorem. This concludes the proof of Theorem~\ref{thm:2}.
\end{proof}

\begin{proof}[\bf Proof of Corollary~\ref{cor:1}]
Throughout this proof, if two sequences $a_n$ and $b_n$ of real numbers satisfy $a_n \leq b_n + \smash{o_{\hspace{0.3mm}\PP_{\!\mathcal{H}_0}}(1)}$, then one writes $a_n \lesssim b_n$ for simplicity. Similarly, if two $p_{\mathcal{U}}\times p_{\mathcal{U}}$ sequences $A_n$ and $B_n$ of matrices with real entries satisfy $A_n = B_n + \smash{o_{\hspace{0.3mm}\PP_{\!\mathcal{H}_0}}(1)} \bb{1}_{p_{\mathcal{U}}} \bb{1}_{p_{\mathcal{U}}}^{\top}$, then one writes $A_n \approx B_n$.

The first step of the proof is to show that $S_{\!\bb{\theta}_0,\mathcal{U},\mathcal{U}} = \EE\big\{\bb{s}_{\mathcal{U}}(\bb{X} \nvert \bb{\theta}_0) \bb{r}_{\mathcal{U}}(\bb{X} \nvert \bb{\theta}_0)^{\top}\big\}\approx R_{\bb{\theta}_0,\mathcal{U}}$. Under the assumptions of the corollary, the maximum likelihood estimator $\bb{\theta}_{\smash{n,\mathcal{U}}}^{\star}$ is asymptotically unbiased and attains the Cram\'er--Rao lower bound in the limit. Consequently, one has, for all $\bb{z}\in \R^{p_{\mathcal{U}}}$,
\[
n \, \bb{z}^{\top} \Var(\bb{\theta}_{\smash{n,\mathcal{U}}}^{\star}) \bb{z} \lesssim n \, \bb{z}^{\top} \Var(T) \bb{z},
\]
for any asymptotically unbiased estimator $T$, including $\hat{\bb{\theta}}_{n,\mathcal{U}}$. Therefore, by adapting the argument of \citet[p.~317]{MR346957} to the present setting, if $C_n = \Cov(\bb{\theta}_{\smash{n,\mathcal{U}}}^{\star}, \hat{\bb{\theta}}_{n,\mathcal{U}} - \bb{\theta}_{\smash{n,\mathcal{U}}}^{\star})$ denotes the cross-covariance matrix, then one finds that, for any nonzero real vector $\bb{z}\in \R^{p_{\mathcal{U}}} \backslash \{\bb{0}_{p_{\mathcal{U}}}\}$ and any scalar $\lambda\in [0, - \bb{z}^{\top} (C_n + C_n^{\top}) \bb{z}/\{\bb{z}^{\top} \Var(\hat{\bb{\theta}}_{n,\mathcal{U}} - \bb{\theta}_{\smash{n,\mathcal{U}}}^{\star}) \bb{z}\}]$,
\[
\begin{aligned}
n \, \bb{z}^{\top} \Var(\bb{\theta}_{\smash{n,\mathcal{U}}}^{\star}) \bb{z}
&\lesssim n \, \bb{z}^{\top} \Var\{\bb{\theta}_{\smash{n,\mathcal{U}}}^{\star} + \lambda (\hat{\bb{\theta}}_{n,\mathcal{U}} - \bb{\theta}_{\smash{n,\mathcal{U}}}^{\star})\} \bb{z} \\
&= n \, \bb{z}^{\top} \Var(\bb{\theta}_{\smash{n,\mathcal{U}}}^{\star}) \bb{z} + \lambda n \, \bb{z}^{\top} (C_n + C_n^{\top}) \bb{z} + \lambda^2 n \, \bb{z}^{\top} \Var(\hat{\bb{\theta}}_{n,\mathcal{U}} - \bb{\theta}_{\smash{n,\mathcal{U}}}^{\star}) \bb{z}
\leq n \, \bb{z}^{\top} \Var(\bb{\theta}_{\smash{n,\mathcal{U}}}^{\star}) \bb{z}.
\end{aligned}
\]
This implies that $n C_n \approx n C_n^{\top} \approx 0_{p_{\mathcal{U}} \times p_{\mathcal{U}}}$, or equivalently using Assumption~\ref{ass:4},
\[
n \, \Cov(\bb{\theta}_{\smash{n,\mathcal{U}}}^{\star}, \hat{\bb{\theta}}_{n,\mathcal{U}})\approx
n \, \Cov(\hat{\bb{\theta}}_{n,\mathcal{U}},\bb{\theta}_{\smash{n,\mathcal{U}}}^{\star})
\approx n \, \Var(\bb{\theta}_{\smash{n,\mathcal{U}}}^{\star})
\approx I_{\bb{\theta}_0,\mathcal{U}}^{-1}.
\]
Moreover, given the expansions in Assumption~\ref{ass:4} for both $\sqrt{n}\,(\bb{\theta}_{\smash{n,\mathcal{U}}}^{\star} - \bb{\theta}_{0,\mathcal{U}})$ and $\sqrt{n}\,(\hat{\bb{\theta}}_{n,\mathcal{U}} - \bb{\theta}_{0,\mathcal{U}})$, one has
\begin{align*}
n \, \Cov(\bb{\theta}_{\smash{n,\mathcal{U}}}^{\star},\hat{\bb{\theta}}_{n,\mathcal{U}})
\approx I_{\bb{\theta}_0,\mathcal{U}}^{-1} \EE\{\bb{s}_{\mathcal{U}}(\bb{X} \nvert \bb{\theta}_0)\bb{r}_{\mathcal{U}}(\bb{X} \nvert \bb{\theta}_0)^{\top}\} R_{\bb{\theta}_0,\mathcal{U}}^{-1}
= I_{\bb{\theta}_0,\mathcal{U}}^{-1} S_{\bb{\theta}_0,\mathcal{U},\mathcal{U}} R_{\bb{\theta}_0,\mathcal{U}}^{-1}.
\end{align*}
Combining the last two equations, one can conclude that $S_{\!\bb{\theta}_0,\mathcal{U},\mathcal{U}} \approx S_{\!\bb{\theta}_0,\mathcal{U},\mathcal{U}}^{\top} \approx R_{\bb{\theta}_0,\mathcal{U}}$.

In a second step, one can now show that $\bb{\delta}_{\mathcal{U}}$ has no impact on the asymptotic distribution of $\sqrt{n} \, \bb{U}_n(\hat{\bb{\theta}}_n)$. If one reruns the proof of Theorem~\ref{thm:2} with $\bb{\delta}$, $\bb{s}$ and $I_{\bb{\theta}_0}$ instead of $\bb{\delta}_{\mathcal{K}}$, $\bb{s}_{\mathcal{K}}$ and $I_{\bb{\theta}_0,\mathcal{K}}$, respectively, then one finds that, as $n \to \infty$,
\[
\sqrt{n} \, \bb{U}_n(\hat{\bb{\theta}}_n) \stackrel{\PP_{\!\mathcal{H}_{1,n}}}{\scalebox{2}[1.2]{$\rightsquigarrow$}} \mathcal{N}_d(M^{\star}_{\bb{\theta}_0}\bb{\delta}, \Sigma_{\bb{\theta}_0,\mathcal{U}}),
\quad \text{with  }
M^{\star}_{\bb{\theta}_0} = \nu \big\{G_{\bb{\theta}_0} - G_{\bb{\theta}_0,\mathcal{U}} R_{\bb{\theta}_0,\mathcal{U}}^{-1} S_{\!\bb{\theta}_0,\mathcal{U}}^{\top}\big\},
\]
where $S_{\!\bb{\theta}_0,\mathcal{U}} = \EE\{\bb{s}(\bb{X} \nvert \bb{\theta})\bb{r}_{\mathcal{U}}(\bb{X} \nvert \bb{\theta})^{\top}\}$. Without loss of generality, assume that the first $p_{\mathcal{K}}$ indices correspond to the known components of $\bb{\theta}_0$ and the last $p_{\mathcal{U}}$ indices correspond to the unknown components of $\bb{\theta}_0$. One decomposes the $p_{\mathcal{U}}\times p$ matrix $S_{\!\bb{\theta}_0,\mathcal{U}}^{\top}$ into two parts, formed of the $p_{\mathcal{U}}\times p_{\mathcal{K}}$ matrix $S_{\!\bb{\theta}_0,\mathcal{K},\mathcal{U}}^{\top}$ and the $p_{\mathcal{U}}\times p_{\mathcal{U}}$ matrix $S_{\!\bb{\theta}_0,\mathcal{U},\mathcal{U}}^{\top} \approx R_{\bb{\theta}_0,\mathcal{U}}$, as $S_{\!\bb{\theta}_0,\mathcal{U}}^{\top} = [S_{\!\bb{\theta}_0,\mathcal{U},\mathcal{K}}^{\top} \, ; \, S_{\!\bb{\theta}_0,\mathcal{U},\mathcal{U}}^{\top}]$. Similarly, one writes $\bb{\delta}^{\top} = [\bb{\delta}_{\mathcal{K}}^{\top} \, ; \, \bb{\delta}_{\mathcal{U}}^{\top}]$ and $G_{\bb{\theta}_0} = [G_{\bb{\theta}_0,\mathcal{K}} \, ; \, G_{\bb{\theta}_0,\mathcal{U}}]$. Then one finds
\begin{align*}
M^{\star}_{\bb{\theta}_0}\bb{\delta}
&= \nu \big\{G_{\bb{\theta}_0} - G_{\bb{\theta}_0,\mathcal{U}} R_{\bb{\theta}_0,\mathcal{U}}^{-1} S_{\!\bb{\theta}_0,\mathcal{U}}^{\top}\big\} \bb{\delta}
\approx \nu \Big\{
\begin{bmatrix}
G_{\bb{\theta}_0,\mathcal{K}} \, ; \, G_{\bb{\theta}_0,\mathcal{U}}
\end{bmatrix}
-
G_{\bb{\theta}_0,\mathcal{U}} R_{\bb{\theta}_0,\mathcal{U}}^{-1}
\begin{bmatrix}
S_{\!\bb{\theta}_0,\mathcal{K},\mathcal{U}}^{\top} \, ; \, R_{\bb{\theta}_0,\mathcal{U}}
\end{bmatrix}
\Big\}
\begin{bmatrix}
\bb{\delta}_{\mathcal{K}} \\[1mm]
\bb{\delta}_{\mathcal{U}}
\end{bmatrix} \\[-2mm]
&= \nu
\begin{bmatrix}
G_{\bb{\theta}_0,\mathcal{K}} - G_{\bb{\theta}_0,\mathcal{U}} R_{\bb{\theta}_0,\mathcal{U}}^{-1} S_{\!\bb{\theta}_0,\mathcal{K},\mathcal{U}}^{\top} \, ; \, \bb{0}_{d\times p_{\mathcal{U}}}
\end{bmatrix}
\begin{bmatrix}
\bb{\delta}_{\mathcal{K}} \\[1mm]
\bb{\delta}_{\mathcal{U}}
\end{bmatrix}
= \nu \big\{G_{\bb{\theta}_0,\mathcal{K}} - G_{\bb{\theta}_0,\mathcal{U}} R_{\bb{\theta}_0,\mathcal{U}}^{-1} S_{\!\bb{\theta}_0,\mathcal{K},\mathcal{U}}^{\top}\big\} \bb{\delta}_{\mathcal{K}}
= M_{\bb{\theta}_0}\bb{\delta}_{\mathcal{K}},
\end{align*}
which only depends on $\bb{\delta}_{\mathcal{K}}$ and corresponds exactly to the expression in Theorem~\ref{thm:2}. Thus the proof is complete.
\end{proof}

\begin{proof}[\bf Proof of Proposition~\ref{prp:1}]
Assumptions~\ref{ass:1}~and~\ref{ass:3} are valid for $\lambda\in (0,\infty)$ and $\lambda\in (1,\infty)$, respectively, as shown by \citet{MR4547729}. Although it does not play a role here, Assumption~\ref{ass:3} was also shown to hold for the more difficult case $\lambda = 1$ (the Laplace distribution under $\mathcal{H}_0$) in that paper; recall Remark~\ref{rem:4}. Assumptions~\ref{ass:2} is automatically satisfied because the kernel function $\bb{h}$ has degree $\nu = 1$. For the maximum likelihood estimator $(\hat{\mu}_n^{\star}, \hat{\sigma}_n^{\star})^{\top}$, Assumption~\ref{ass:4} is satisfied by setting $\bb{r}_{\mathcal{U}}(x \nvert \bb{\theta}_0) = \bb{s}_{\mathcal{U}}(x \nvert \bb{\theta}_0)$ and $R_{\bb{\theta}_0,\mathcal{U}} = I_{\bb{\theta}_0,\mathcal{U}}$, as mentioned in Remark~\ref{rem:3}. Given that Assumptions~\ref{ass:1}--\ref{ass:4} hold, the results of Theorems~\ref{thm:1}~and~\ref{thm:2} apply.

It remains to compute the expression of the matrices $\Sigma_{\bb{\theta}_0,\mathcal{U}}$ and $M_{\bb{\theta}_0}$. Note that
\[
\begin{aligned}
H_{\bb{\theta}_0}
&= \EE\big\{\bb{h}(X \nvert \bb{\theta}_0) \bb{h}(X \nvert \bb{\theta}_0)^{\top}\big\}
= \EE\big\{\bb{s}_{\mathcal{K}}(X \nvert \bb{\theta}_0)\bb{s}_{\mathcal{K}}(X \nvert \bb{\theta}_0)^{\top}\big\}
= I_{\bb{\theta}_0,\mathcal{K}}, \\
G_{\bb{\theta}_0,\mathcal{U}}
&= \EE\big\{\bb{h}(X \nvert \bb{\theta}_0) \bb{s}_{\mathcal{U}}(X \nvert \bb{\theta}_0)^{\top}\big\}
= \EE\big\{\bb{s}_{\mathcal{K}}(X \nvert \bb{\theta}_0) \bb{s}_{\mathcal{U}}(X \nvert \bb{\theta}_0)^{\top}\big\}
= I_{\bb{\theta}_0,\mathcal{K},\mathcal{U}}, \\
G_{\bb{\theta}_0,\mathcal{K}}
&= \EE\big\{\bb{h}(X \nvert \bb{\theta}_0) \bb{s}_{\mathcal{K}}(X \nvert \bb{\theta}_0)^{\top}\big\}
= \EE\big\{\bb{s}_{\mathcal{K}}(X \nvert \bb{\theta}_0)\bb{s}_{\mathcal{K}}(X \nvert \bb{\theta}_0)^{\top}\big\}
= I_{\bb{\theta}_0,\mathcal{K}}
= H_{\bb{\theta}_0}.
\end{aligned}
\]
Letting $y = (x-\mu_0)/\sigma_0$, straightforward calculations in \texttt{Mathematica} yield
\begin{equation}\label{eq:vector.d}
\begin{aligned}
&\bb{s}_{\mathcal{K}}(x \nvert \bb{\theta}_0)
=
\begin{bmatrix}
\partial_{\theta_1} \ln \{ f_{\lambda}(x \nvert\bb{\theta}_0)\} \\[1mm]
\partial_{\theta_2} \ln \{ f_{\lambda}(x \nvert\bb{\theta}_0)\}
\end{bmatrix}
=
\begin{bmatrix}
-2 |y|^{\lambda} \mathrm{sign}(y) \\[1mm]
-\frac{1}{\lambda} \Big[|y|^{\lambda} \ln |y| - \frac{1}{\lambda} \left\{\ln (\lambda) + \psi(1+1/\lambda)\right\}\Big]
\end{bmatrix}, \\
&\bb{s}_{\mathcal{U}}(x \nvert \bb{\theta}_0)
=
\begin{bmatrix}
\partial_{\mu} \ln \{ f_{\lambda}(x \nvert\bb{\theta}_0)\} \\[1mm]
\partial_{\sigma} \ln \{f_{\lambda}(x \nvert\bb{\theta}_0)\}
\end{bmatrix}
= \frac{1}{\sigma_0}
\begin{bmatrix}
|y|^{\lambda - 1} \mathrm{sign}(y) \\[1mm]
|y|^{\lambda} - 1
\end{bmatrix}.
\end{aligned}
\end{equation}
Using the expressions in \eqref{eq:vector.d}, one also verifies with \texttt{Mathematica} that, for $X\sim \mathrm{EPD}_{\lambda}(\mu_0, \sigma_0)$,
\begin{equation}
\label{eq:info.matrices}
\begin{aligned}
&I_{\bb{\theta}_0,\mathcal{K}}
=
\begin{bmatrix}
4(1 + \lambda) & 0 \\[1mm]
0 & \lambda^{-3} (C_{1,\lambda}+C_{2,\lambda}^2)
\end{bmatrix}, \;
%%%%%%%%
I_{\bb{\theta}_0,\mathcal{U}}
= \frac{1}{\sigma_0^2}
\begin{bmatrix}
\frac{\lambda^{1-2/\lambda}\Gamma(2-1/\lambda)}{\Gamma(1+1/\lambda)} & 0 \\[1mm]
0 & \lambda
\end{bmatrix}, \;
%%%%%%%%
I_{\bb{\theta}_0,\mathcal{K},\mathcal{U}}
= \frac{1}{\sigma_0}
\begin{bmatrix}
-\frac{2\lambda^{1-1/\lambda}}{\Gamma(1+1/\lambda)} & 0 \\[1mm]
0 & -\lambda^{-1} C_{2,\lambda}
\end{bmatrix},
\end{aligned}
\end{equation}
where $C_{1,\lambda} = (1+1/\lambda)\psi_1(1+1/\lambda)-1$ and $C_{2,\lambda} = 1+ \ln (\lambda) + \psi (1+1/\lambda)$.

Given Remark~\ref{rem:3} and the above matrices, one obtains
\[
\Sigma_{\bb{\theta}_0,\mathcal{U}}
= H_{\bb{\theta}_0} - G_{\bb{\theta}_0,\mathcal{U}} I_{\smash{\bb{\theta}_0,\mathcal{U}}}^{-1} G_{\smash{\bb{\theta}_0,\mathcal{U}}}^{\top}
= I_{\bb{\theta}_0,\mathcal{K}} - I_{\bb{\theta}_0,\mathcal{K},\mathcal{U}} I_{\smash{\bb{\theta}_0,\mathcal{U}}}^{-1} I_{\bb{\theta}_0,\mathcal{K},\mathcal{U}}^{\top}
=
\begin{bmatrix}
4(1 + \lambda) - \frac{4\lambda}{\Gamma(2 - 1/\lambda) \Gamma(1+1/\lambda)} & 0 \\
0 & \lambda^{-3} C_{1,\lambda}
\end{bmatrix}.
\]
Similarly, according to Remark~\ref{rem:6} and the above matrices, one finds
\[
M_{\bb{\theta}_0}
= G_{\bb{\theta}_0,\mathcal{K}} - G_{\bb{\theta}_0,\mathcal{U}} I_{\bb{\theta}_0,\mathcal{U}}^{-1} I_{\!\bb{\theta}_0,\mathcal{K}, \mathcal{U}}^{\top}
= I_{\bb{\theta}_0,\mathcal{K}} - I_{\bb{\theta}_0,\mathcal{K},\mathcal{U}} I_{\bb{\theta}_0,\mathcal{U}}^{-1} I_{\!\bb{\theta}_0,\mathcal{K}, \mathcal{U}}^{\top}
= \Sigma_{\bb{\theta}_0,\mathcal{U}}.
\]
This concludes the proof.
\end{proof}

\begin{proof}[\bf Proof of Proposition~\ref{prp:2}]
As in the proof of Proposition~\ref{prp:1}, Assumptions~\ref{ass:1}--\ref{ass:3} hold for $\lambda\in (1,\infty)$. However, finding functions $\bb{r}_{\mathcal{U}}(x \nvert \bb{\theta}_0)$ and $R_{\bb{\theta}_0,\mathcal{U}}$ that satisfy Assumption~\ref{ass:4} for the method of moments estimator is not so obvious. Let $X_1,\ldots, X_n$ be a random sample from $\mathrm{EPD}_{\lambda}(\mu_0,\sigma_0)$. By applying a Taylor expansion, one has
\begin{equation*}
\sqrt{n}
\begin{bmatrix}
\hat{\mu}_n-\mu_0 \\[1mm]
\hat{\sigma}_n -\sigma_0
\end{bmatrix}
=
\frac{1}{\sqrt{n}}\sum_{i=1}^n \bb{g}_{\mathcal{U}}(X_i \nvert \bb{\theta}_0)+o_{\hspace{0.3mm}\PP_{\!H_0}}(1) \bb{1}_2, \quad
\bb{g}_{\mathcal{U}}(x \nvert \bb{\theta}_0)
=
\begin{bmatrix}
x - \mu_0 \\[1mm]
\frac{1}{2\sigma_0}\left\{\frac{3\Gamma(1+1/\lambda)}{\lambda^{2/\lambda}\Gamma(1+3/\lambda)}(x - \mu_0)^2 - \sigma_0^2\right\}
\end{bmatrix},
\end{equation*}
where $\EE\big\{\bb{g}_{\mathcal{U}}(X \nvert \bb{\theta}_0)\big\} = \bb{0}_2$ for $X \sim \mathrm{EPD}_{\lambda}(\mu_0,\sigma_0)$. Now, one obtains $\EE\{g_{\mathcal{U},1}(X \nvert \bb{\theta}_0)g_{\mathcal{U},2}(X \nvert \bb{\theta}_0)\} = 0$ because the map $y\mapsto g_{\mathcal{U},1}(\mu_0 + \sigma_0 y \nvert \bb{\theta}_0)g_{\mathcal{U},2}(\mu_0 + \sigma_0 y \nvert \bb{\theta}_0)$ is antisymmetric and the $\mathrm{EPD}_{\lambda}(0,1)$ density is symmetric. Also, by invoking Lemma~\ref{lem:2} in \ref{app:A}, one finds
\[
 \EE\{g_{\mathcal{U},1}(X \nvert \bb{\theta}_0)^2\}
= \EE\{(X -\mu_0)^2\}
= \frac{\lambda^{2/\lambda}\Gamma(1+3/\lambda)}{3\Gamma(1+1/\lambda)} \, \sigma_0^2,
\]
and
\[
\begin{aligned}
\EE\{g_{\mathcal{U},2}(X \nvert \bb{\theta}_0)^2\}
&=\frac{1}{4\sigma_0^2} \EE\left[\left\{\frac{3\Gamma(1+1/\lambda)}{\lambda^{2/\lambda}\Gamma(1+3/\lambda)}(X - \mu_0)^2 - \sigma_0^2\right\}^2\right] \\
&= \frac{9\Gamma^2(1+1/\lambda)}{\lambda^{4/\lambda}\Gamma^2(1+3/\lambda)4\sigma_0^2} \EE\{(X-\mu_0)^4\}
+\frac{\sigma_0^2}{4}-\frac{3\Gamma(1+1/\lambda)}{2\lambda^{2/\lambda}\Gamma(1+3/\lambda)} \EE\{(X-\mu_0)^2\} \\
&= \frac{9\Gamma^2(1+1/\lambda)}{\lambda^{4/\lambda}\Gamma^2(1+3/\lambda)4\sigma_0^2}
\frac{\lambda^{4/\lambda}\Gamma(1+5/\lambda)\sigma_0^4}{5\Gamma(1+1/\lambda)}+\frac{\sigma_0^2}{4}-\frac{\sigma_0^2}{2}
= \frac{9\Gamma(1+1/\lambda)\Gamma(1+5/\lambda)\sigma_0^2}{20\Gamma^2(1+3/\lambda)}-\frac{\sigma_0^2}{4} \\
&= \frac{9\Gamma(1+1/\lambda)\Gamma(1+5/\lambda)-\Gamma^2(1+3/\lambda)}{5\Gamma^2(1+3/\lambda)} \frac{\sigma_0^2}{4}
= \frac{1}{4C_{3,\lambda}} \, \sigma_0^2.
\end{aligned}
\]
One deduces that Assumption~\ref{ass:4} is satisfied with
\begin{equation}\label{eq:R.inv}
R_{\bb{\theta}_0,\mathcal{U}}^{-1}
= \EE\big\{\bb{g}_{\mathcal{U}}(X \nvert \bb{\theta}_0) \bb{g}_{\mathcal{U}}(X \nvert \bb{\theta}_0)^{\top}\big\}
= \sigma_0^2
\begin{bmatrix}
\frac{\lambda^{2/\lambda}\Gamma(1+3/\lambda)}{3\Gamma(1+1/\lambda)} & 0 \\[1mm]
0 & \frac{1}{4C_{3,\lambda}}
\end{bmatrix}
\end{equation}
and
\[
\bb{r}_{\mathcal{U}}(x \nvert \bb{\theta}_0)
= R_{\bb{\theta}_0,\mathcal{U}}\bb{g}_{\mathcal{U}}(x \nvert \bb{\theta}_0)=\frac{1}{\sigma_0}
\begin{bmatrix}
\frac{3\Gamma(1+1/\lambda)}{\lambda^{2/\lambda}\Gamma(1+3/\lambda)} \left(\frac{x - \mu_0}{\sigma_0}\right) \\[1mm]
2C_{3,\lambda}\left\{\frac{3\Gamma(1+1/\lambda)}{\lambda^{2/\lambda}\Gamma(1+3/\lambda)}\left(\frac{x - \mu_0}{\sigma_0}\right)^2 - 1\right\}
\end{bmatrix}.
\]
Given that Assumptions~\ref{ass:1}--\ref{ass:4} hold, the results of Theorems~\ref{thm:1}~and~\ref{thm:2} apply.

It remains to compute the expressions for the matrices $\Sigma_{\bb{\theta}_0,\mathcal{U}}$ and $M_{\bb{\theta}_0}$. One has $\EE\{s_{1,\mathcal{K}}(X \nvert \bb{\theta}_0)r_{2,\mathcal{U}}(X \nvert \bb{\theta}_0)\} = 0$ because the map $y\mapsto s_{1,\mathcal{K}}(\mu_0 + \sigma_0 y \nvert \bb{\theta}_0)r_{2,\mathcal{U}}(\mu_0 + \sigma_0 y \nvert \bb{\theta}_0)$ is antisymmetric and the $\mathrm{EPD}_{\lambda}(0,1)$ density is symmetric. Also, using Lemma~\ref{lem:2} in \ref{app:A} with $Y=(X-\mu_0)/\sigma_0$, one finds
\begin{align*}
\sigma_0 \, \EE\left\{s_{1,\mathcal{K}}(X \nvert \bb{\theta}_0)r_{1,\mathcal{U}}(X \nvert \bb{\theta}_0)\right\}
&= \EE\left\{-2 |Y|^{\lambda} \mathrm{sign}(Y)\frac{3\Gamma(1+1/\lambda)}{\lambda^{2/\lambda}\Gamma(1+3/\lambda)}Y\right\}
= \frac{-6 \Gamma(1+1/\lambda)}{\lambda^{2/\lambda}\Gamma(1+3/\lambda)} \, \EE\big(|Y|^{\lambda+1}\big) \\
&= \frac{-6 \Gamma(1+1/\lambda)}{\lambda^{2/\lambda}\Gamma(1+3/\lambda)}\frac{\lambda^{1/\lambda}\Gamma(1+2/\lambda)}{\Gamma(1 + 1/\lambda)}
= \frac{-6 \Gamma(1+2/\lambda)}{\lambda^{1/\lambda}\Gamma(1+3/\lambda)},
\end{align*}
and
\begin{align*}
&\sigma_0 \, \EE\left\{s_{2,\mathcal{K}}(X \nvert \bb{\theta}_0)r_{2,\mathcal{U}}(X \nvert \bb{\theta}_0)\right\} \\
&\quad= -\frac{1}{\lambda} \, 2C_{3,\lambda} \EE\left[ \Big[|Y|^{\lambda} \ln |Y| - \frac{1}{\lambda} \left\{\ln (\lambda) + \psi(1+1/\lambda)\right\}\Big]\left\{\frac{3\Gamma(1+1/\lambda)}{\lambda^{2/\lambda}\Gamma(1+3/\lambda)}Y^2 - 1\right\}\right] \\
&\quad= -\frac{1}{\lambda} \, 2C_{3,\lambda} \EE\left[|Y|^{\lambda} \ln |Y| \left\{\frac{3\Gamma(1+1/\lambda)}{\lambda^{2/\lambda}\Gamma(1+3/\lambda)}Y^2 - 1\right\}\right] \\
&\quad= -\frac{1}{\lambda} \, 2C_{3,\lambda}\left\{\frac{3\Gamma(1+1/\lambda)}{\lambda^{2/\lambda}\Gamma(1+3/\lambda)} \EE(|Y|^{2+\lambda} \ln |Y|)- \EE(|Y|^{\lambda} \ln |Y|)\right\} \\
&\quad= -\frac{1}{\lambda} \, 2C_{3,\lambda}\left[\frac{3\Gamma(1+1/\lambda)}{\lambda^{2/\lambda}\Gamma(1+3/\lambda)}
\frac{\lambda^{2/\lambda-1}\Gamma(1+3/\lambda)\{\ln (\lambda) + \psi(1+3/\lambda)\}}{\Gamma(1 + 1/\lambda)}
-\frac{\lambda^{-1} \, \Gamma(1+1/\lambda)\left\{\ln (\lambda) + \psi(1+1/\lambda)\right\}}{\Gamma(1 + 1/\lambda)}\right] \\
&\quad= -\frac{1}{\lambda}\, 2C_{3,\lambda}\left[3\lambda^{-1}\{\ln (\lambda) + \psi(1+3/\lambda)\}-\lambda^{-1}\{\ln (\lambda) + \psi(1+1/\lambda)\}\right] \\
&\quad= -2\lambda^{-2} \, C_{3,\lambda}\left\{2\ln (\lambda) + 3\psi(1+3/\lambda) - \psi(1+1/\lambda)\right\},
\end{align*}
so that
\[
\begin{aligned}
J_{\bb{\theta}_0,\mathcal{U}}
&= \EE\big\{\bb{h}(X \nvert \bb{\theta}_0) \bb{r}_{\mathcal{U}}(X \nvert \bb{\theta}_0)^{\top}\big\}
= \EE\big\{\bb{s}_{\mathcal{K}}(X \nvert \bb{\theta}_0)\bb{r}_{\mathcal{U}}(X \nvert \bb{\theta}_0)^{\top}\big\} \\
&=
\frac{1}{\sigma_0}
\begin{bmatrix}
\frac{-6 \Gamma(1+2/\lambda)}{\lambda^{1/\lambda}\Gamma(1+3/\lambda)} & 0 \\[1mm]
0 & -2\lambda^{-2}C_{3,\lambda} \left\{2\ln (\lambda) + 3\psi(1+3/\lambda) - \psi(1+1/\lambda)\right\}
\end{bmatrix}.
\end{aligned}
\]
Furthermore, using the expression for the expected information submatrices in \eqref{eq:info.matrices}, one has
\[
\begin{aligned}
H_{\bb{\theta}_0}
&= \EE\big\{\bb{h}(X \nvert \bb{\theta}_0) \bb{h}(X \nvert \bb{\theta}_0)^{\top}\big\}
= \EE\big\{\bb{s}_{\mathcal{K}}(X \nvert \bb{\theta}_0)\bb{s}_{\mathcal{K}}(X \nvert \bb{\theta}_0)^{\top}\big\}
= I_{\bb{\theta}_0,\mathcal{K}}
=
\begin{bmatrix}
4(1 + \lambda) & 0 \\[1mm]
0 & \lambda^{-3} (C_{1,\lambda}+C_{2,\lambda}^2)
\end{bmatrix}, \\
G_{\bb{\theta}_0,\mathcal{U}}
&= \EE\big\{\bb{h}(X \nvert \bb{\theta}_0) \bb{s}_{\mathcal{U}}(X \nvert \bb{\theta}_0)^{\top}\big\}
= \EE\big\{\bb{s}_{\mathcal{K}}(X \nvert \bb{\theta}_0) \bb{s}_{\mathcal{U}}(X \nvert \bb{\theta}_0)^{\top}\big\}
= I_{\bb{\theta}_0,\mathcal{K},\mathcal{U}}
= \frac{1}{\sigma_0}
\begin{bmatrix}
-\frac{2\lambda^{1-1/\lambda}}{\Gamma(1+1/\lambda)} & 0 \\[1mm]
0 & -\lambda^{-1} C_{2,\lambda}
\end{bmatrix},
\end{aligned}
\]
and given the expression for $R_{\bb{\theta}_0,\mathcal{U}}^{-1}$ in \eqref{eq:R.inv}, straightforward calculations yield
\begin{align*}
\Sigma_{\bb{\theta}_0,\mathcal{U}}
&= H_{\bb{\theta}_0} - G_{\bb{\theta}_0,\mathcal{U}} R_{\bb{\theta}_0,\mathcal{U}}^{-1} J_{\bb{\theta}_0,\mathcal{U}}^{\top}
- J_{\bb{\theta}_0,\mathcal{U}} R_{\bb{\theta}_0,\mathcal{U}}^{-1} G_{\bb{\theta}_0,\mathcal{U}}^{\top}
+ G_{\bb{\theta}_0,\mathcal{U}} R_{\bb{\theta}_0,\mathcal{U}}^{-1} G_{\bb{\theta}_0,\mathcal{U}}^{\top} \\
&=
\begin{bmatrix}
4(1 + \lambda) + \frac{4\lambda^2\Gamma(1+3/\lambda)}{3\Gamma^3(1+1/\lambda)}-\frac{8\lambda\Gamma(1+2/\lambda)}{\Gamma^2(1+1/\lambda)} & 0 \\[1mm]
0 & \lambda^{-3} \left(C_{1,\lambda}+2C_{2,\lambda}^2 + \frac{\lambda}{4} \, C_{2,\lambda}^2 C_{3,\lambda}^{-1} - C_{2,\lambda}C_{4,\lambda}\right)
\end{bmatrix}.
\end{align*}
Similarly, in view of the fact that
\[
\begin{aligned}
G_{\bb{\theta}_0,\mathcal{K}}
&= \EE\big\{\bb{h}(X \nvert \bb{\theta}_0) \bb{s}_{\mathcal{K}}(X \nvert \bb{\theta}_0)^{\top}\big\}
= \EE\big\{\bb{s}_{\mathcal{K}}(X \nvert \bb{\theta}_0)\bb{s}_{\mathcal{K}}(X \nvert \bb{\theta}_0)^{\top}\big\}
= I_{\bb{\theta}_0,\mathcal{K}}
= H_{\bb{\theta}_0}, \\
S_{\!\bb{\theta}_0,\mathcal{K}, \mathcal{U}}
&= \EE\big\{\bb{s}_{\mathcal{K}}(X \nvert \bb{\theta}_0) \bb{r}_{\mathcal{U}}(X \nvert \bb{\theta}_0)^{\top}\big\}
= J_{\bb{\theta}_0,\mathcal{U}},
\end{aligned}
\]
and given the above expressions for $G_{\bb{\theta}_0,\mathcal{U}}$, $H_{\bb{\theta}_0}$ and $J_{\bb{\theta}_0,\mathcal{U}}$, together with $R_{\bb{\theta}_0,\mathcal{U}}^{-1}$ in \eqref{eq:R.inv}, straightforward calculations yield
\[
M_{\bb{\theta}_0} = H_{\bb{\theta}_0} - G_{\bb{\theta}_0,\mathcal{U}} R_{\bb{\theta}_0,\mathcal{U}}^{-1} J_{\bb{\theta}_0,\mathcal{U}}^{\top}
=
\begin{bmatrix}
4(1 + \lambda)-\frac{4\lambda\Gamma(1+2/\lambda)}{\Gamma^2(1+1/\lambda)} & 0 \\
0 & \lambda^{-3} \left(C_{1,\lambda} + \frac{3}{2} \, C_{2,\lambda}^2- \frac{1}{2} \, C_{2,\lambda} C_{4,\lambda}\right)
\end{bmatrix}.
\]
This concludes the proof.
\end{proof}

\begin{proof}[\bf Proof of Proposition~\ref{prp:3}]
First, Assumption~\ref{ass:1} is satisfied because the kernel has mean zero if $X_1,X_2\sim \mathrm{SN}(\mu,\sigma^2,0) \equiv \mathcal{N}(\mu,\sigma^2)$, for all $\mu\in \R$ and $\sigma^2 \in (0,\infty)$. Second, for $X_2\sim \mathrm{SN}(\bb{\theta}_0) \equiv \mathcal{N}(\mu_0,\sigma_0^2)$, one has
\[
h^{(1)}(x \nvert \bb{\theta}_0) = \EE(\ind_{\{x + X_2 \geq 2 \mu_0\}} - 1/2) = \Phi\left(\frac{x - \mu_0}{\sigma_0}\right) - 1/2.
\]
Given that $\EE\{h^{(1)}(X \nvert \bb{\theta}_0)^2\}$ is positive and finite for $X\sim \mathrm{SN}(\bb{\theta}_0)$, the kernel $h$ is non-degenerate under $\mathcal{H}_0$, i.e., Assumption~\ref{ass:2} is satisfied. By contiguity \eqref{eq:contiguity} and Assumption~\ref{ass:2}, it can be assumed from hereon, without loss of generality, that the kernel has degree $1$ and is defined by
\[
h(x \nvert \bb{\theta}) = \nu \, h^{(1)}(x \nvert \bb{\theta}) = 2 \, \Phi\left(\frac{x - \mu}{\sigma}\right) - 1.
\]
Under this definition, Assumptions~\ref{ass:3}~(a)~and~(b) are trivial. To prove Assumption~\ref{ass:3}~(c), recall that $f(x \nvert \mu, \sigma^2, \alpha) = 2 \sigma^{-1} \phi\{(x-\mu)/\sigma\} \Phi\{\alpha (x-\mu)/\sigma\}$. Thus one has, for all $x \in \R$,
\[
\begin{aligned}
\frac{\partial}{\partial \mu} f(x \nvert \bb{\theta})
&= \frac{(x - \mu)}{\sigma^2} f(x \nvert \bb{\theta}) - \frac{2\alpha}{\sigma^2} \phi\left(\frac{x - \mu}{\sigma}\right) \phi\left(\alpha \, \frac{x - \mu}{\sigma}\right), \\
%%%
\frac{\partial}{\partial \sigma^2} f(x \nvert \bb{\theta})
&= \left\{ -\frac{1}{2 \sigma^2} + \frac{(x - \mu)^2}{2 \sigma^4} \right\} f(x \nvert \bb{\theta}) - \alpha \, \frac{(x - \mu)}{\sigma^4} \phi\left(\frac{x - \mu}{\sigma}\right) \phi\left(\alpha \, \frac{x - \mu}{\sigma}\right), \\
%%%
\frac{\partial}{\partial \alpha} f(x \nvert \bb{\theta})
&= 2 \frac{(x - \mu)}{\sigma^2} \phi\left(\frac{x - \mu}{\sigma}\right) \phi\left(\alpha \, \frac{x - \mu}{\sigma}\right).
\end{aligned}
\]

Given that the kernel $h$ is bounded by $2$, then Assumption~\ref{ass:3}~(c) is proved by taking
\[
L_1(x) = \max_{\bb{\theta}\in B_{\sigma_0^2/2}(\bb{\theta}_0)} \left|\frac{\partial}{\partial \mu} f(x \nvert \bb{\theta})\right|, \quad
L_2(x) = \max_{\bb{\theta}\in B_{\sigma_0^2/2}(\bb{\theta}_0)} \left|\frac{\partial}{\partial \sigma^2} f(x \nvert \bb{\theta})\right|, \quad
L_3(x) = \max_{\bb{\theta}\in B_{\sigma_0^2/2}(\bb{\theta}_0)} \left|\frac{\partial}{\partial \alpha} f(x \nvert \bb{\theta})\right|.
\]
Indeed, the set $B_{\sigma_0^2/2}(\bb{\theta}_0)$ is compact in $\Theta$, and the radius of the ball, $\sigma_0^2/2$, is chosen so that $\sigma^2$ always remains bounded away from $0$. Therefore, the functions $L_1$, $L_2$, and $L_3$ are integrable, because even if one takes the parameters in $B_{\sigma_0^2/2}(\bb{\theta}_0)$ which yield the worst tail behavior, they are still Gaussian-like (or better). To see this, note that all factors of the form $\phi\{\alpha (x_i - \mu)/\sigma\}$ and $\Phi\{\alpha (x_i - \mu)/\sigma\}$ in the partial derivatives are bounded uniformly, and the remaining factors $\phi\{(x - \mu)/\sigma\}$, multiplied by polynomials of any degree in $(x - \mu)/\sigma$, are still integrable given that all Gaussian moments exist. To sum up, Assumption~\ref{ass:3}~(c) is satisfied.

Finally, Assumption~\ref{ass:4} is well-known to be satisfied for the maximum likelihood estimator \eqref{eq:example.2.MLE} in this setting, with the score and expected information matrix for the unknown components being, at $\bb{\theta} = \bb{\theta}_0$, respectively equal to
\[
\bb{r}_{\mathcal{U}}(x \nvert \bb{\theta}_0) = \bb{s}_{\mathcal{U}}(x \nvert \bb{\theta}_0) =
\begin{bmatrix}
(x - \mu_0) / \sigma_0^2 \\[1mm]
-1/(2 \sigma_0^2) + (x - \mu_0)^2/(2 \sigma_0^4)
\end{bmatrix},
\]
and
\[
R_{\bb{\theta}_0,\mathcal{U}} = I_{\bb{\theta}_0,\mathcal{U}} =
\begin{bmatrix}
1/\sigma_0^2 & 0 \\
0 & 1/(2\sigma_0^4)
\end{bmatrix};
\]
see, e.g., \citet[p.121]{MR1699953}, and recall Remark~\ref{rem:3}.

Now that Assumptions~\ref{ass:1}--\ref{ass:4} have been verified, it remains to compute the expressions of $\Sigma_{\bb{\theta},\mathcal{U}}$ and $M_{\bb{\theta}_0}$ in the statements of Theorems~\ref{thm:1}~and~\ref{thm:2}. According to Remarks~\ref{rem:3}~and~\ref{rem:6}, in the maximum likelihood setting, the expressions for $\Sigma_{\bb{\theta}_0,\mathcal{U}}$ and $M_{\bb{\theta}_0}$ simplify to
\[
\Sigma_{\bb{\theta}_0,\mathcal{U}} = \nu^2 (H_{\bb{\theta}_0} - G_{\bb{\theta}_0,\mathcal{U}} I_{\smash{\bb{\theta}_0,\mathcal{U}}}^{-1} G_{\smash{\bb{\theta}_0,\mathcal{U}}}^{\top}), \quad
M_{\bb{\theta}_0}
= \nu \, \big(G_{\bb{\theta}_0,\mathcal{K}} - G_{\bb{\theta}_0,\mathcal{U}} I_{\bb{\theta}_0,\mathcal{U}}^{-1} I_{\!\bb{\theta}_0,\mathcal{K}, \mathcal{U}}^{\top}\big),
\]
where recall that $G_{\bb{\theta}} = \EE\big\{h^{(1)}(X \nvert \bb{\theta}) \bb{s}(X \nvert \bb{\theta})^{\top}\big\}$ for $X\sim \mathrm{SN}(\mu,\sigma^2,\alpha)$. Here, the degree of the $U$-statistic is $\nu = 2$. For simplicity of notation, let $Z\sim \mathcal{N}(0,1)$. Using the above expressions for $\bb{s}_{\mathcal{U}}(\cdot \nvert \bb{\theta}_0)$ and $I_{\bb{\theta}_0,\mathcal{U}}$, one has, for $X\sim \mathrm{SN}(\mu_0,\sigma_0^2,0) \equiv \mathcal{N}(\mu_0,\sigma_0^2)$,
\[
\begin{aligned}
H_{\bb{\theta}_0}
&= \EE\left[\left\{\Phi\left(\frac{X - \mu_0}{\sigma_0}\right) - 1/2\right\}^2\right] = \EE\{\Phi^2(Z)\} - \EE\{\Phi(Z)\} + \frac{1}{4} = \frac{1}{3} - \frac{1}{2} + \frac{1}{4} = \frac{1}{12}, \\[1mm]
G_{\bb{\theta}_0,\mathcal{U}}
&=
\begin{bmatrix}
\sigma_0^{-1} \EE[Z \{\Phi(Z) - 1/2\}] \\[1mm]
-(2 \sigma_0^2)^{-1} \EE\{\Phi(Z) - 1/2\} + (2 \sigma_0^2)^{-1} \EE[Z^2 \{\Phi(Z) - 1/2\}]
\end{bmatrix}^{\top}
=
\begin{bmatrix}
1 / (2 \sigma_0 \sqrt{\pi}) \\
0
\end{bmatrix}^{\top}.
\end{aligned}
\]
These calculations are straightforward applications of integration by parts and were verified with \texttt{Mathematica}. In turn, one finds that
\[
\Sigma_{\bb{\theta}_0,\mathcal{U}} = 2^2 \left(\frac{1}{12} - \frac{1}{2 \sigma_0 \sqrt{\pi}} \times \sigma_0^2 \times \frac{1}{2 \sigma_0 \sqrt{\pi}}\right) = \frac{1}{3} - \frac{1}{\pi}.
\]

Furthermore, one has
\[
s_{\mathcal{K}}(x \nvert \bb{\theta}_0)
= \left.\frac{\partial}{\partial \alpha} \log f(x \nvert \bb{\theta})\right|_{\bb{\theta} = \bb{\theta}_0} = \frac{x - \mu_0}{\sigma_0} \times \frac{\phi\{\alpha_0 (x - \mu_0) / \sigma_0\}}{\Phi\{\alpha_0 (x - \mu_0) / \sigma_0\}} = \frac{x - \mu_0}{\sigma_0} \sqrt{\frac{2}{\pi}},
\]
and thus
\[
\begin{aligned}
G_{\bb{\theta}_0,\mathcal{K}}
&= \sqrt{\frac{2}{\pi}} \, \EE\left[\left(\frac{X - \mu_0}{\sigma_0}\right) \left\{\Phi\left(\frac{X - \mu_0}{\sigma_0}\right) - 1/2\right\}\right] = \sqrt{\frac{2}{\pi}} \, \EE[Z \{\Phi(Z) - 1/2\}] = \frac{1}{\pi \sqrt{2}}, \\[1mm]
\mathcal{I}_{\bb{\theta}_0,\mathcal{K},\mathcal{U}}^{\top}
&=
\EE\{s_{\mathcal{K}}(x \nvert \bb{\theta}_0) s_{\mathcal{U}}(x \nvert \bb{\theta}_0)\}
=
\begin{bmatrix}
\frac{1}{\sigma_0} \sqrt{\frac{2}{\pi}} \, \EE\left\{\left(\frac{X - \mu_0}{\sigma_0}\right)^2\right\} \\
\frac{-1}{2 \sigma_0^2} \sqrt{\frac{2}{\pi}} \, \EE\left(\frac{X - \mu_0}{\sigma_0}\right) + \frac{1}{2 \sigma_0^2} \sqrt{\frac{2}{\pi}} \, \EE\left\{\left(\frac{X - \mu_0}{\sigma_0}\right)^3\right\}
\end{bmatrix}
=
\begin{bmatrix}
\frac{1}{\sigma_0} \sqrt{\frac{2}{\pi}} \\[1mm]
0
\end{bmatrix}.
\end{aligned}
\]
It follows that
\[
M_{\bb{\theta}_0} = \nu \, \big(G_{\bb{\theta}_0,\mathcal{K}} - G_{\bb{\theta}_0,\mathcal{U}} I_{\bb{\theta}_0,\mathcal{U}}^{-1} I_{\!\bb{\theta}_0,\mathcal{K}, \mathcal{U}}^{\top}\big) = 2 \, \left(\frac{1}{\pi \sqrt{2}} - \frac{1}{2 \sigma_0 \sqrt{\pi}} \times \sigma_0^2 \times \frac{1}{\sigma_0} \sqrt{\frac{2}{\pi}}\right) = 0.
\]
The application of Theorems~\ref{thm:1}~and~\ref{thm:2} with $\Sigma_{\bb{\theta}_0,\mathcal{U}} = 1/3 - 1/\pi$ and $M_{\bb{\theta}_0} = 0$ concludes the proof.
\end{proof}

\begin{appendices}

\renewcommand{\thesection}{Appendix~\Alph{section}}

\section{Technical lemmas}\label{app:A}

The first lemma below follows from a small adaptation of a well-known uniform law of large numbers due to Le Cam \citep{LeCam1952phd}. Its proof follows the strategy described in Section 16 of \citet{MR1699953}, with a small adaptation to treat the case where the parameter space is not compact and the $U$-statistic has arbitrary degree. One also assumes that the estimator $\hat{\bb{\theta}}_n$ converges in $\PP$-probability to $\bb{\theta}_0$ instead of almost-surely, which leads to a weakened conclusion of uniform convergence in probability.

\begin{lem}\label{lem:1}
Let $\bb{X}_1, \bb{X}_2, \ldots$ be a sequence of iid random variables with common distribution function $F(\cdot \nvert \bb{\theta}_0)$, and let the estimator $\hat{\bb{\theta}}_n = \hat{\bb{\theta}}_n(\bb{X}_1, \ldots, \bb{X}_n)$ be convergent in $\PP$-probability to $\bb{\theta}_0\in \R^p$. Assume that $g: \R^m \times \dots \times \R^m \times \R^p \to \R$ is a Borel measurable function and that there exists a strictly positive real $\delta\in (0, \infty)$ such that $B_\delta (\boldsymbol{\theta}_0)\subseteq \Theta$ and
\begin{itemize}\setlength\itemsep{0em}

\item[\upshape(i)]
for all $\bb{x}_1, \ldots, \bb{x}_{\nu}\in \R^m$, the function $\bb{\theta}\mapsto g(\bb{x}_1, \ldots, \bb{x}_{\nu} \nvert \bb{\theta})$ is continuous on $B_{\delta}(\bb{\theta}_0) = \{\bb{\theta}\in \R^p : \|\bb{\theta} - \bb{\theta}_0\|_2 \leq \delta\}$;

\item[\upshape(ii)]
there exists a nonnegative function $K : \R^m \times \dots \times \R^m \to [0, \infty)$ such that $|g(\bb{x}_1, \ldots, \bb{x}_{\nu} \nvert \bb{\theta})| \leq K(\bb{x}_1, \ldots, \bb{x}_{\nu})$ for all $(\bb{x}_1, \ldots, \bb{x}_{\nu}, \bb{\theta})\in \R^m \times \dots \times \R^m \times B_{\delta}(\bb{\theta}_0)$, and $ \EE\{K(\bb{X}_1, \ldots, \bb{X}_{\nu})\} < \infty$;

\item[\upshape(iii)]
the kernel $g$ is non-degenerate, i.e., for $\bb{X}_1, \ldots, \bb{X}_n\stackrel{\mathrm{iid}}{\sim} F(\cdot \nvert \bb{\theta}_0)$, as $n \to \infty$,
\[
\sqrt{n} \, \binom{n}{\nu}^{-1} \sum_{(n, \nu)} g(\bb{X}_{i_1}, \ldots, \bb{X}_{i_{\nu}} \nvert \bb{\theta}_0) = \nu \frac{1}{\sqrt{n}} \sum_{i=1}^n g^{(1)}(\bb{X}_i \nvert \bb{\theta}_0) + O_{\PP_{\!\mathcal{H}_0}}\left(\frac{1}{\sqrt{n}}\right),
\]
where $g^{(1)}(\bb{x} \nvert \bb{\theta}_0) = \EE\{g(\bb{x}, \bb{X}_2, \ldots, \bb{X}_{\nu} \nvert \bb{\theta}_0)\}$ is a kernel of degree 1 and $ \EE[\{g^{(1)}(\bb{X} \nvert \bb{\theta}_0)\}^2]$ is positive.
\end{itemize}
Then, for $\rho_n = \|\hat{\bb{\theta}}_n - \bb{\theta}_0\|_2$, one has, as $n \to \infty$,
\[
M_n \equiv \sup_{\bb{\theta} \in B_{\rho_n}(\bb{\theta}_0)} \left|\binom{n}{\nu}^{-1} \sum_{(n, \nu)} g(\bb{X}_{i_1}, \ldots, \bb{X}_{i_{\nu}} \nvert \bb{\theta}) - \EE\{g(\bb{X}_1, \ldots, \bb{X}_{\nu} \nvert \bb{\theta}_0)\}\right| \stackrel{\PP}{\longrightarrow} 0.
\]
\end{lem}

\begin{proof}[\bf Proof of Lemma~\ref{lem:1}]
Let a real $\varepsilon\in (0,\infty)$ be given. By (i), (ii) and Lebesgue's dominated convergence theorem, one knows that the map $\bb{\theta}\mapsto \EE\{g(\bb{X}_1, \ldots, \bb{X}_{\nu} \nvert \bb{\theta})\}$ is uniformly continuous on the compact ball $B_{\delta}(\bb{\theta}_0)$, so choose $\delta' = \delta'(\varepsilon)\in (0, \delta)$ small enough that
\[
\sup_{\bb{\theta}\in B_{\delta'}(\bb{\theta}_0)} \big| \EE\{g(\bb{X}_1, \ldots, \bb{X}_{\nu} \nvert \bb{\theta})\} - \EE\{g(\bb{X}_1, \ldots, \bb{X}_{\nu} \nvert \bb{\theta}_0)\}\big| < {\varepsilon}/{2}.
\]
A union bound then yields
\[
\PP(M_n > \varepsilon) \leq \PP\big\{B_{\rho_n}(\bb{\theta}_0) \not\subseteq B_{\delta'}(\bb{\theta}_0)\big\} + \PP\left[\sup_{\bb{\theta}\in B_{\delta'}(\bb{\theta}_0)} \bigg|\binom{n}{\nu}^{-1} \sum_{(n, \nu)} g(\bb{X}_{i_1}, \ldots, \bb{X}_{i_{\nu}} \nvert \bb{\theta}) - \EE\{g(\bb{X}_1, \ldots, \bb{X}_{\nu} \nvert \bb{\theta})\}\bigg| > \frac{\varepsilon}{2}\right].
\]
On the right-hand side, the first summand tends to zero as $n\to\infty$ by the assumption that $\hat{\bb{\theta}}_n$ converges in $\PP$-probability to $\bb{\theta}_0$. Under assumptions (i), (ii) and (iii), the second summand tends to zero by Theorem~16~(a) of \citet{MR1699953}.
\end{proof}

The second lemma states results about the moments of the $\mathrm{APD}_{\lambda}(1/2,\lambda,0,1)$ distribution defined in \eqref{eq:APD.density}. These facts are given with proof in Equations~(B.25)~and~(B.26) of the Supplementary~Material~B of \citet{MR4547729}. This information is reproduced here for the readers' convenience.

\begin{lem}\label{lem:2}
Let the reals $\lambda\in [1,\infty)$ and $a\in (-1,\infty)$ be given. For $Y \sim \mathrm{APD}_{\lambda}(1/2,\lambda,0,1)$, with $Y=(X-\mu_0)/\sigma_0$ and $X \sim \mathrm{APD}_{\lambda}(1/2,\lambda,\mu_0,\sigma_0)$, one has
\[
\EE\left(|Y|^a\right)
= \frac{\lambda^{a/\lambda-1}\Gamma\{(a+1)/\lambda\}}{\Gamma(1 + 1/\lambda)}, \quad
\EE\left(|Y|^a\ln |Y|\right)
= \frac{\lambda^{a/\lambda-2}\Gamma\{(a+1)/\lambda\}\big[\ln (\lambda) + \psi\{(a+1)/\lambda\}\big]}{\Gamma(1 + 1/\lambda)},
\]
where, for arbitrary real $x \in (0,\infty)$, $\Gamma(x) = \int_0^{\infty} t^{x-1} e^{-t} \rd t$ denotes Euler's gamma function and $\psi(x) = \rd\ln \{ \Gamma (x)\}/\rd x$ denotes the digamma function.
\end{lem}

\end{appendices}

\section*{Funding}
\addcontentsline{toc}{section}{Funding}

Genest's research is funded in part by the Canada Research Chairs Program (Grant no.~950--231937) and the Natural Sciences and Engineering Research Council of Canada (RGPIN-2024-04088). Part of this work was carried out while Ouimet was a CRM-Simons postdoctoral fellow at the Centre de recherches math\'ematiques (Montr\'eal, Canada). His current funding as a postdoctoral researcher at McGill University was made possible through a contribution to Genest's research program from the Trottier Institute for Science and Public Policy.

\section*{Acknowledgments}
\addcontentsline{toc}{section}{Acknowledgments}

The authors thank the anonymous reviewers for their thoughtful and constructive feedback. The detailed comments and suggestions offered by the referees have helped improve both the quality and depth of this paper.

\section*{Supplementary material}
\addcontentsline{toc}{section}{Supplementary material}

Supplementary material available at the journal's website contains the \textsf{R} code to reproduce the examples and graphs contained in Section~\ref{sec:5}.

%\section*{References}
\addcontentsline{toc}{section}{References}

\bibliographystyle{myjmva}
\bibliography{bib}
\end{document}